\documentclass[12pt]{article}
\usepackage{lineno,hyperref}

\usepackage{makecell}
\usepackage{amsmath}
\usepackage{graphicx}
\usepackage{subfigure}
\usepackage{amsfonts}
\usepackage{verbatim}
\usepackage{latexsym}
\usepackage{epstopdf}
\usepackage{float}
\usepackage{color}
\usepackage{enumitem}
\usepackage{amssymb}
\usepackage{amsthm}
\usepackage{a4}
\allowdisplaybreaks[4]

\makeatletter
\@addtoreset{equation}{section}

\begin{document}

\newcommand{\EE}{\mathbb{E}}
\newcommand{\PP}{\mathbb{P}}
\newcommand{\RR}{\mathbb{R}}
\newcommand{\SM}{\mathbb{S}}
\newcommand{\ZZ}{\mathbb{Z}}
\newcommand{\ind}{\mathbf{1}}
\newcommand{\LL}{\mathbb{L}}
\def\F{{\cal F}}
\def\G{{\cal G}}
\def\P{{\cal P}}

\newtheorem{theorem}{Theorem}[section]
\newtheorem{lemma}[theorem]{Lemma}
\newtheorem{coro}[theorem]{Corollary}
\newtheorem{defn}[theorem]{Definition}
\newtheorem{assp}[theorem]{Assumption}
\newtheorem{cond}[theorem]{Condition}
\newtheorem{expl}[theorem]{Example}
\newtheorem{prop}[theorem]{Proposition}
\newtheorem{rmk}[theorem]{Remark}
\newtheorem{conj}[theorem]{Conjecture}

\newcommand\tq{{\scriptstyle{3\over 4 }\scriptstyle}}
\newcommand\qua{{\scriptstyle{1\over 4 }\scriptstyle}}
\newcommand\hf{{\textstyle{1\over 2 }\displaystyle}}
\newcommand\hhf{{\scriptstyle{1\over 2 }\scriptstyle}}
\newcommand\hei{\tfrac{1}{8}}

\newcommand{\eproof}{\indent\vrule height6pt width4pt depth1pt\hfil\par\medbreak}

\def\proof{\noindent{\it Proof. }}
\def\Proof{\noindent{\it Proof of Theorem \ref{thm:invar}. }}

\def\refer{\hangindent=0.3in\hangafter=1}

\newcommand\wD{\widehat{\D}}
\newcommand{\ka}{\kappa_{10}}

\title{ Explicit positivity preserving numerical method for linear stochastic volatility models driven by $\alpha$-stable process}

\author{Xiaotong Li\footnotemark[2] , Wei Liu\footnotemark[2] , Xuerong Mao\footnotemark[4] , Hongjiong Tian\footnotemark[2] , Yue Wu\footnotemark[4]}

\maketitle

\renewcommand{\thefootnote}{\fnsymbol{footnote}}
\footnotetext[2]{Department of Mathematics, Shanghai Normal University, Shanghai, 200234, China.}
\footnotetext[4]{Department of Mathematics and Statistics, University of Strathclyde, Glasgow, G1 1XH, UK.}

\begin{abstract}
In this paper, we introduce a linear stochastic volatility model driven by $\alpha$-stable processes, which admits a unique positive solution. To preserve positivity, we modify the classical forward Euler-Maruyama scheme and analyze its numerical properties. The scheme achieves a strong convergence order of $1/\alpha$.  Numerical simulations are presented at the end to verify theoretical results.
\end{abstract}



\section{Introduction}
Stochastic volatility models, which represent a new generation of option pricing models, gained significant attention following their introduction in 1987 \cite{HullWhite1987,Scott1987,Wiggins1987}. These models highlighted the inherent risks in securities markets and the highly dynamic nature of such risks. For a more detailed exploration of these models, readers are encouraged to consult the comprehensive works in \cite{Bergomi2015,Lewis2000}.

A typical example of a stochastic volatility model is the GARCH diffusion model introduced in the monograph \cite{Lewis2000} (GARCH is a loose term that can accommodate many types of discrete-time financial models, and may have various continuous-time constraints).  As being the continuous time limit of many GARCH type processes,  the GARCH diffusion, driven by a Brownian motion $B$, represents an actual volatility process of the form, 
\begin{equation}\label{model}
dV_{t}=(a-bV_{t}) dt+c V_{t} dB_{t},
\end{equation}
where the volatility drift parameters, $a$ and $b$, are assumed to be constants, capturing the mean-reverting nature of the volatility process, and $c$ represents the volatility uncertainty. Since $b$ has the dimensions of inverse time, $1/b$ represents a ``half-life" for volatility shocks.

Brownian motion is characterized by normally distributed increments, which decay exponentially in their tails. This makes it inadequate for modeling real-world phenomena that exhibit large jumps with tail distributions following a power-law pattern \cite{CaoOuyangLiLiChen2018,Taleb2020arXiv}, where jump-diffusion L\'evy processes provide a more suitable characterization. For example, the Ait-Sahalia-type rate model with Poisson jumps are explored in \cite{Coffie2024,LeiGanLiu2023,ZhaoWangWang2021}.

The $\alpha$-stable process is a special process of the L\'evy processes. Its particularity is that there is no $q$-th moment when $q\geq\alpha$. These distinct mathematical properties make them particularly effective for simulating and modeling real-world phenomena with infinite variance, or frequent, unpredictable large jumps \cite{KRSSB2009,NikiasShao1995,Pierce1997}. In finance, Fu and Li \cite{FuLi2010}, Li and Ma \cite{LiMa2015} studied the asymptotic properties of the Cox-Ingersoll-Ross (CIR) model driven by an $\alpha$-stable process. Jiao, Ma and Scotti \cite{JiaoMaScotti2017} revealed that the $\alpha$-CIR model can describe several recent observations of the sovereign bond market, such as the persistency of low interest rates together with the presence of large jumps. 

As stochastic volatility is often significantly impacted by sudden events, it is intuitive to use $\alpha$-stable noise for modelling purpose. In this paper, we propose to replace Brownian motion $B_t$ in the stochastic volatility model \eqref{model}  with an $\alpha$-stable process $L^{\alpha}(t)$. We demonstrate such a stochastic differential equation (SDE) is well-posed in the sense that it admits a unique positive solution with probability one (see Theorem \ref{thm:possol}).

We are also interested in developing appropriate numerical treatments for the stochastic volatility model driven by $\alpha$-stable process as the solution is typically not available in closed form. In Section \ref{sec4}, we propose a modified Euler-Maruyama (EM) method to simulate the solution while preserving the positivity. 
In literature, EM has been shown to well approximate time-homogeneous SDEs driven by symmetric $\alpha$-stable noise, even if the drift coefficients are not irregular. For instance, the performance of EM for additive SDEs in the case of $\gamma$-H\"older continuous ($0<\gamma\leq 1$) and bounded drift coefficients has been exhaustively studied: from the initial attempt by Pamen and Taguchi \cite{PamenTaguchi2017}, which for the first time established a strong convergence order $(\gamma/2)\wedge(1/p)$ for EM in the $L_p~(p\geq 1)$ sense, to the recent work by Butkovsky, Dareiotis and Gerencs{\'e}r \cite{BDG2024}, which lifted the order of strong $L_p$-convergence to $1/2+\gamma/\alpha-\epsilon$ for an arbitrary $\epsilon\in (0,1/2)$ within a relaxed range $\alpha\in[2/3,2]$ and in the entire range of $\gamma $, i.e., $ \gamma>1- \alpha/2$. When confronted with multiplicative $\alpha$-stable process, Mikulevi\v cius and Xu \cite[Propositions 1--2]{MikuleviciusXu2018} derived strong convergence orders in the entire range of $\gamma$.

In financial applications of SDEs where solution positivity is essential, carefully adapted EM algorithms have been developed to ensure positivity is maintained.
Li and Taguchi \cite{LiTaguchi2019} proposed a positivity-preserving implicit EM scheme for jump-extended CIR processes, where jumps are controlled by a compensated spectrally positive $\alpha$-stable process for $\alpha\in(1,2)$. Inspired by the design of \cite{LiTaguchi2019}, Li and Liu \cite{LiLiu2023arXiv} constructed a partially-implicit positivity-preserving EM for a jump-extended constant elasticity of variance process. We therefore stress that, to the best of our knowledge, this paper is the first one that considers modified explicit EM for SDEs driven by $\alpha$-stable processes that also preserves positivity. The strong order of convergence is confirmed in Theorem \ref{thm:conver2} to be $1/\alpha$ in $L_q$-norm, for $q\in [1,\alpha)$.

The structure of the rest of the paper is as follows. Section \ref{sec2} presents the mathematics preliminaries. Section \ref{sec3} discusses the existence and uniqueness of the positive global solution to our model. Section \ref{sec4} gives the positivity preserving EM scheme, proves its convergence, and provides the convergence rate. Numerical simulations supporting the theoretical results are presented in Section \ref{sec5}. Section \ref{sec6} presents the conclusion and outlines directions for future work.

\section{Mathematical preliminaries}\label{sec2}
Let $(\Omega,\{\mathcal{F}_t\}_{t\geq 0},\PP)$ be a complete probability space with a filtration $\{\mathcal{F}_t\}_{t\geq 0}$ satisfying the usual conditions (i.e., it is right continuous and increasing in $t$ while $\mathcal{F}_0$ contains all $\PP$-null sets). Let $\EE$ denote the probability expectation with respect to $\PP$. Moreover, define $a\vee b=\max(a,b)$ and $a\wedge b=\min(a,b)$ for any $a,b\in \RR$.

A random variable $ X $ is said to follow a stable distribution, denoted by $ X\sim S_\alpha(\sigma, \beta, \mu) $, if it has characteristic function of the following form,
\begin{align*}
\varphi_X(u) &= \mathbb{E} \left[ \exp\{iuX\} \right] \\
&= \left\{
\begin{array}{ll}
	\exp \left\{ -\sigma^\alpha |u|^\alpha \left( 1 - i\beta \textup{sgn}(u) \tan \frac{\alpha\pi}{2} \right) \right\} + i\mu u, & \text{if } \alpha \neq 1, \\
	\exp \left\{ -\sigma |u| \left( 1 + i\beta \frac{2}{\pi} \textup{sgn}(u) \log |u| \right) \right\} + i\mu u, & \text{if } \alpha = 1,
\end{array}
\right.
\end{align*}
where $ \alpha \in (0, 2]$ is the index of stability, $\sigma \in (0, \infty)$ is the scale parameter, $\beta \in [-1, 1]$ is the skewness parameter and $ \mu \in (-\infty, \infty)$ is the location parameter. When $ \mu = 0 $, we say $ X $ is strictly $ \alpha $-stable ($ \alpha \neq 1 $). We refer to \cite{SamorodnitskyTaqqu1994} for more details on stable distributions.

$L^{\alpha}(t)$ is a scalar $\alpha$-stable process with $\alpha\in(1,2)$. There are equivalent definitions of the $\alpha$-stable processes, such as by using the L\'evy–Khinchine formula or by using the L\'evy–It\^o decomposition \cite{SamorodnitskyTaqqu1994}. For a more comprehensive introduction to L\'evy processes and $\alpha$-stable processes, please refer to the monographs \cite{D.Applebaum2009,JanickiWeron2021,Sato2013}.  Here, we use the following one as it is convenient for the simulation. A stochastic process $L^{\alpha}(t)$ is called the strict $\alpha$-stable process if
\begin{itemize}
\item $L^{\alpha}(0)=0$, a.s.;
\item For any $m\in N$ and $0\leq t_1<t_2< \ldots t_m\leq T$, the random variables $L^{\alpha}(t_0), L^{\alpha}(t_1)-L^{\alpha}(t_0), L^{\alpha}(t_2)-L^{\alpha}(t_1), \ldots, L^{\alpha}(t_m)-L^{\alpha}(t_{m-1})$ are independent;
\item For any $0\leq s<t<\infty$, $L^{\alpha}(t)-L^{\alpha}(s)$ follows $S_{\alpha}((t-s)^{1/\alpha}, \beta, 0)$, where $S_{\alpha}(\sigma, \beta, \mu)$ is a four-parameter stable distribution.
\end{itemize}

In this paper, we consider scalar stochastic volatility models of the following form
\begin{equation}\label{mainSDE}
d x(t) = \left( \mu - \lambda x(t) \right) dt + \kappa x(t-) d L^{\alpha}(t),~~~x(0)=x_0.
\end{equation}
where $x(t-)$ denotes the left-hand limit of $x$ at time $t$, and $L^{\alpha}(t)$ is an $\alpha$-stable process for $\alpha \in (1,2)$ with the L\'evy measure defined as
\begin{equation}\label{nu}
\nu(dz) = \frac{C_{\alpha}}{|z|^{\alpha+1}}dz, \quad \text{for}~z\neq0,
\end{equation}
where $C_{\alpha}=\frac{\alpha 2^{\alpha-1}\Gamma(\frac{\alpha+1}{2})}{\pi^{\frac{1}{2}}\Gamma(1-\frac{\alpha}{2})}$, $\Gamma(\cdot)$ is the gamma function. By the L\'evy-It\^o decomposition, an $\alpha$-stable process admits the following integral representation, 
\begin{equation*}
L^{\alpha}(t)=\int_{0}^{t}\int_{0<|z|\leq 1}z \widetilde{N}(dz,ds)+\int_{0}^{t}\int_{|z|>1}z N(dz,ds),
\end{equation*}
where $N$ is the Poisson measure on $[0,\infty)\times(\RR \backslash \{0\})$  with  $\EE N(dz,dt)=\nu (dz)dt$ and $\widetilde{N}(dz,dt)=N(dz,dt)-\nu(dz)dt$ is the compensated Poisson random measure, with its L\'evy measure denoted by $\nu$.

Throughout the whole paper, we suppose that the following assumptions hold.
\begin{assp}\label{assp:coef}
The parameters in \eqref{mainSDE} satisfy $\mu>1$, $\lambda>0$, $ 0<\kappa<1 $ and
\begin{equation*}
 \lambda > \frac{2\kappa^{0.5}C_{\alpha}}{2\alpha-1}. 
\end{equation*}
\end{assp}

In practical applications, it is reasonable to impose a limit on the jump size. For instance, the maximum downward jump should not exceed the current price of the stock, and so on. Based on this, we give the following assumption.

\begin{assp}\label{assp:levy}
The L\'evy measure $\nu(dz)$ is allowed to have negative jumps, but we assume there is a lower bound for negative jumps, that is, the jump height $z$ satisfies
\begin{equation*}
z>-\frac{1}{\kappa}.
\end{equation*}
\end{assp}

The following three important lemmas are crucial for proving our main results.
\begin{lemma}\label{Ito}
For a function $V(x)\in C^2(\RR)$, the It\^o formula \textup{\cite{D.Applebaum2009,OksendalSulem2000}} of $V(x)$ associated with SDE \eqref{mainSDE} is defined by
\begin{align*}
V(x(t))&=V(x(0))+\int_{0}^{t}V'(x(s))f(x(s))ds\\
&\quad +\int_{0}^{t}\int_{\RR\backslash\{0\}}\Big(V(x(s-)+g(x(s-))z)-V(x(s-))\Big)\widetilde{N}(ds,dz)\\
&\quad +\int_{0}^{t}\int_{\RR\backslash\{0\}}\Big(V(x(s-)+g(x(s-))z) -V(x(s-))\\
&\qquad \qquad \qquad\quad  -V'(x(s-))g(x(s-))z\mathbb{I}_{\{|z|\leq 1\}}(z)\Big)\nu(dz)ds,
\end{align*}
where $f(x)=\mu-\lambda x$ and $g(x)=\kappa x$.
\end{lemma}
\begin{lemma}\label{L_alpha_moment}\textup{\cite{SamorodnitskyTaqqu1994}}
Let $L^{\alpha}(t)\sim S_{\alpha}(t^{1/\alpha},\beta,0)$ with $\alpha\in(1,2)$, for any $q\in[1,\alpha)$, there is a constant $\bar{C}$ such that
\begin{equation*}
	\EE|L^{\alpha}(t)|^q\leq \bar{C}t^{\frac{q}{\alpha}}.
\end{equation*}
\end{lemma}

\begin{lemma}\label{L_alpha_moment_2}\textup{\cite{Zhang2013}}
Let $\{\xi(t)\}_{t\geq 0}$ be a left continuous $(\mathcal{F}_t)$-adapted $\RR^d$-valued process and that satisfies 
\begin{equation*}
\int_{0}^{T}\EE|\xi(s)|^{\alpha}ds<+\infty, ~~\forall~T>0.
\end{equation*}
Then there exists a constant $C_{p,\alpha}>0$, for any $p\in(0,\alpha)$, 
\begin{equation*}
\EE\left[\sup_{0\leq t \leq T}\left|\int_{0}^{t}\xi(s)dL^{\alpha}(s)\right|^p\right]\leq C_{p,\alpha}\left(\int_{0}^{T}\EE|\xi(s)|^{\alpha}ds\right)^{\frac{p}{\alpha}}.
\end{equation*}
\end{lemma}

\begin{rmk}
In this paper, we use the positive constant $C$, which is independent of $\Delta$, to denote a generic constant that may vary at different instances.
\end{rmk}

\section{Properties of the underlying solution}\label{sec3}
In this section, we first prove that SDE \eqref{mainSDE} has a unique positive solution $x(t)$. We then prove the boundedness of the $q$th moments of the underlying solution $x(t)$ for $q\in [1,\alpha)$.
\begin{theorem}\label{thm:possol}
Let Assumptions \ref{assp:coef} and \ref{assp:levy} hold.
For any initial value $x_0 > 0$, the SDE \eqref{mainSDE} has a unique global positive solution for all $t > 0$ with probability one, i.e.
\begin{equation}\label{x_positive}
\PP \left( x(t) \in \RR_+~\text{for all}~t>0 \right)=1.
\end{equation}
\end{theorem}
\begin{proof}
Let $R_0>0$ be sufficiently large such that $R_0>x_0$. For each integer $R>R_0$, define a stopping time
\begin{equation}\label{tauR}
\tau_R := \inf \left\{ t >0: x(t) \not\in (R^{-1},R)\right\}.
\end{equation}
To prove the assertion, we need to show that for any fixed $T > 0$, $\PP(\tau_{\infty} \leq T) = \lim_{R \rightarrow \infty} \PP(\tau_R \leq T) = 0$.
\par
Choose a $C^2(\RR_+,\RR_+)$ function
\begin{equation*}
V(x) = x^{0.5} - 1 - 0.5\log x,
\end{equation*}
with the corresponding derivatives given by 
\begin{equation*}
V'(x)=\frac{1}{2}x^{-0.5}-\frac{1}{2}x^{-1} ~~\text{and}~~V''(x)=-\frac{1}{4}x^{-1.5}+\frac{1}{2}x^{-2}.
\end{equation*}
By the It\^o formula, we have
\begin{equation}\label{posithm:Vterm}
V(x(t \wedge \tau_R)) = V(x_0) + J_1  + J_2 + M_t^1,
\end{equation}
where
\begin{align*}
J_1 &= \int_0^{t \wedge \tau_R} V'(x(s))\left(\mu - \lambda x(s) \right)ds \\
&= \frac{1}{2} \int_0^{t \wedge \tau_R} \left( \mu x^{-0.5}(s) - \lambda x^{0.5}(s) - \mu x^{-1}(s) + \lambda \right) ds,
\end{align*}
\begin{align*}
J_2 &=  \int_0^{t \wedge \tau_R} \int_{\RR \backslash \{0\}}\Big( V(x(s-)+\kappa x(s-)z) - V(x(s-)) \\
&\qquad \qquad\qquad\qquad - \kappa x(s-)z V'(x(s-)) \mathbb{I}_{\{|z| \leq 1\}}(z)   \Big) \nu(dz) ds\\
&=\int_0^{t \wedge \tau_R} \int_{0<|z|\leq 1}\Big( V(x(s-)+\kappa x(s-)z) - V(x(s-)) - \kappa x(s-)z V'(x(s-))   \Big) \nu(dz) ds \\
&\quad +\int_0^{t \wedge \tau_R} \int_{|z|>1}\Big( V(x(s-)+\kappa x(s-)z) - V(x(s-))    \Big) \nu(dz) ds \\
&= J_{21} + J_{22},
\end{align*}
and $M_t^1$ is a local martingale of the form
\begin{equation*}
M_t^1 =  \int_0^{t \wedge \tau_R} \int_{\RR \backslash \{0\}}\Big( V(x(s-)+\kappa x(s-)z) - V(x(s-))   \Big) \widetilde N (ds,dz).
\end{equation*}
To estimate $J_{21}$, applying Taylor's expansion
\begin{equation*}
V(x + \Delta x) = V(x) + V'(x)\Delta x + (\Delta x)^2 \int_0^1 (1 - \theta)V''(x+\theta \Delta x) d \theta,
\end{equation*}
we have
\begin{align*}
J_{21}&= \kappa^2 \int_0^{t \wedge \tau_R} x^{2}(s-) \int_{0<|z|\leq 1} z^2 \int_0^1 (1-\theta) \bigg(-\frac{1}{4} \left( x(s-)+ \theta \kappa z x(s-) \right)^{-1.5} \\
&\quad + \frac{1}{2}\left( x(s-)+ \theta \kappa z x(s-) \right)^{-2 }  \bigg) d \theta \nu(dz) ds.
\end{align*}
Since $0<\kappa <1$ and $x(s-) > 0$ for $s \in (0,t \wedge \tau_R)$, 
\begin{equation*}
-\frac{1}{4} \left( x(s-)+ \theta \kappa z x(s-) \right)^{-1.5}  \leq 0.
\end{equation*}
Thus, we can see
\begin{align*}
J_{21} &< \kappa^2 \int_0^{t \wedge \tau_R} x^{2}(s-) \int_{0<|z|\leq 1} z^2 \int_0^1 \frac{1-\theta}{2}(1+\theta\kappa z)^{-2} x^{-2}(s-) d \theta \nu(dz) ds \\
&<\kappa^2\int_0^{t \wedge \tau_R}x^{2}(s-)\int_{-1}^{0}\frac{1}{2}z^2\frac{\kappa z-\log(1+\kappa z)}{\kappa^2z^2}x^{-2}(s-)\nu(dz)ds\\
&\quad +\kappa^2\int_{0}^{t\wedge\tau_R}x^2(s-)\int_{0}^{1}z^2\int_{0}^{1}\frac{(1-\theta)}{2}x^{-2}(s-)d\theta\nu(dz)ds\\
&=\frac{1}{2}\int_{-1}^{0}\left(\kappa z-\log(1+\kappa z)\right)\frac{C_{\alpha}}{(-z)^{\alpha+1}}dz\int_0^{t \wedge \tau_R} 1 ds+\kappa^2\int_{0}^{1}z^2\frac{1}{4}\frac{C_{\alpha}}{z^{\alpha+1}}dz\int_0^{t \wedge \tau_R} 1 ds,
\end{align*}
where the equality above is obtained by \eqref{nu} and calculating the integrals in terms of $\theta$. Let $y=\kappa z$, then we have
\begin{align*}
\int_{-1}^{0}\left(\kappa z-\log(1+\kappa z)\right)\frac{C_{\alpha}}{(-z)^{\alpha+1}}dz&=\kappa^{\alpha}\int_{-\kappa}^{0}\frac{y-\log(1+y)}{(-y)^2}\frac{C_{\alpha}}{(-y)^{\alpha-1}}dy\\
&\leq  \frac{C_{\alpha}\left(-\kappa-\log(1-\kappa)\right)}{2-\alpha},
\end{align*}
where we used \eqref{nu} and the fact that
\begin{equation*}
\frac{y-\log(1+y)}{(-y)^2}\leq \frac{-\kappa-\log(1-\kappa)}{\kappa^2},~~y\in[-\kappa,0).
\end{equation*}
Therefore, we have
\begin{equation}\label{J21}
J_{21}<\left(\frac{C_{\alpha}\left(-\kappa-\log(1-\kappa)\right)}{2(2-\alpha)}+\frac{\kappa^2C_{\alpha}}{4(2-\alpha)}\right)\int_0^{t \wedge \tau_R} 1 ds.
\end{equation}
\par
To estimate $J_{22}$, by Assumption \ref{assp:levy}  and the definition of $V(x)$ we have
\begin{align*}
J_{22} &= \int_0^{t \wedge \tau_R} \int_{|z|>1}\left( V(x(s-)+\kappa x(s-)z) - V(x(s-))  \right) \nu(dz) ds\\
&=\int_{0}^{t\wedge \tau_R}\int_{|z|>1}\left(\left(x(s-)+\kappa x(s-)z\right)^{0.5}-x^{0.5}(s-)-\frac{1}{2}\log(1+\kappa z)\right)\nu(dz) ds\\
&< \int_{0}^{t\wedge\tau_R}\int_{-\frac{1}{\kappa}}^{-1}\left(\left(x(s-)+\kappa x(s-)z\right)^{0.5}-x^{0.5}(s-)\right)\nu(dz)ds\\
&\quad + \int_{0}^{t\wedge\tau_R}\int_{1}^{+\infty}\left(\left(x(s-)+\kappa x(s-)z\right)^{0.5}-x^{0.5}(s-)\right)\nu(dz)ds\\
&\leq \int_{0}^{t\wedge\tau_R}\int_{-\frac{1}{\kappa}}^{-1}x^{0.5}(s-)\left((1+\kappa z)^{0.5}-1\right)\nu(dz)ds\\
&\quad +\int_{0}^{t\wedge\tau_R}\int_{1}^{+\infty}\kappa^{0.5}x^{0.5}(s-)z^{0.5}\nu(dz)ds\\
&\leq \int_{0}^{t\wedge\tau_R}\int_{1}^{+\infty}\kappa^{0.5}x^{0.5}(s-)z^{0.5}\nu(dz)ds,
\end{align*}
where we used the facts that  $x(s-)>0$ for $s \in (0, t \wedge\tau_R)$ and
\begin{equation*}
(1+\kappa z)^{0.5}-1<0, ~~\text{where}~~z\in(-\frac{1}{\kappa},-1),~\kappa\in(0,1).
\end{equation*}
Then we can get
\begin{equation*}
J_{22}<\frac{2\kappa^{0.5}C_{\alpha}}{2\alpha-1}\int_0^{t \wedge \tau_R}x^{0.5}(s-) ds,
\end{equation*}
where the inequality above is obtained by calculating the integral with respect to $z$ with $\nu(dz) = C_{\alpha} |z|^{-\alpha -1}dz$.
\par
Thanks to Assumption \ref{assp:coef}, from the definition of $J_1$ and the estimates of $J_{12}$ and $J_{22}$ above, we can see that terms with the largest and the smallest powers are $\left(-\lambda+\frac{2\kappa^{0.5}C_{\alpha}}{2\alpha-1} \right)x^{0.5}(s)$ and $-\mu x^{-1}(s)$. Since both of the coefficients of those two terms are negative, by the basic property of polynomials we can conclude
\begin{equation*}
J_1 + J_2 < \int_0^{t \wedge \tau_R} C ds,
\end{equation*}
where $C$ is a constant dependent on $\mu, \lambda, \kappa $.
\par
Taking expectations on both sides of \eqref{posithm:Vterm} and using the fact $\EE[M_t^1] = 0$ yields
\begin{equation}\label{posithm:Vest}
\EE \left[V(x(T \wedge \tau_R)) \right] < V(x_0) + C\EE[T\wedge \tau_R]\leq V(x_0) + CT.
\end{equation}
It is not hard to see that $V(x)$ is decreasing for $x \in (0,1)$ and increasing for $x \in [1, +\infty)$. In addition, from the definition of the stopping time $\tau_R$ we have
\begin{equation*}
x(\tau_R) \geq R \quad \text{or} \quad x(\tau_R) \leq R^{-1}.
\end{equation*}
Now, we have the following estimate
\begin{equation*}
V(x(\tau_R)) \geq V(R) \wedge V(R^{-1}).
\end{equation*}
Using \eqref{posithm:Vest} together with the estimate that
\begin{equation*}
\PP(\tau_R \leq T) \left( V(R) \wedge V(R^{-1}) \right) \leq \EE \left[\mathbb{I}_{\{\tau_R \leq T\}} V(x(\tau_R))\right] \leq  \EE \left[V(x(T\wedge \tau_R))\right],
\end{equation*}
we obtain
\begin{equation*}
\PP(\tau_R \leq T) \leq \frac{V(x_0) +CT}{V(R) \wedge V(R^{-1})}.
\end{equation*}
Letting $R \rightarrow \infty$ gives $\PP(\tau_{\infty} \leq T) = 0$, which completes the proof.  \eproof
\end{proof}
The next lemma is to show the moment boundedness of the solution to \eqref{mainSDE}.

\begin{lemma}\label{lemma:1}
For any $q \in [1, \alpha)$, there exists a positive constant C such that
\begin{equation*}
\sup_{t \in [0,T]} \EE \left[x^{q}(t) \right] \leq C, \quad \text{for any} ~T>0.
\end{equation*}
\end{lemma}
\begin{proof}
From Theorem \ref{thm:possol}, we know the solution stays positive for any $t>0$. So, for some $R >x_0$ we define a stopping time
\begin{equation*}
\eta_R := \inf \left\{ t >0: x(t) \geq R\right\}.
\end{equation*}
For any fixed $T > 0$ and $q \in [1, \alpha)$, by the It\^o formula we have that for any $t\in [0,T]$
\begin{equation}\label{qmomentthm:meq}
x^q(t\wedge \eta_R) = x_0^q + I_1 + I_2 + M_t^2,
\end{equation}
where
\begin{equation*}
I_1 := \int_0^{t \wedge \eta_R}  qx^{q-1}(s)(\mu-\lambda x(s)) ds=\int_0^{t \wedge \eta_R} \left( q\mu x^{q-1}(s) - q\lambda x^{q}(s)\right) ds,
\end{equation*}
\begin{align*}
I_2 &:=  \int_0^{t \wedge \eta_R} \int_{0<|z|\leq 1} \left( \left( x(s-) + \kappa x(s-)z \right)^q  - x^q(s-) - q x^{q-1}(s-) \kappa x(s-) z \right) \nu(dz) ds \\
&\quad + \int_0^{t \wedge \eta_R} \int_{|z|>1} \left( \left( x(s-) + \kappa x(s-)z \right)^q - x^q(s-) \right) \nu(dz) ds \\
&=: I_{21} + I_{22},
\end{align*}
and $M_t^2$ is a local martingale of the form
\begin{equation*}
M_t^2 :=  \int_0^{t \wedge \eta_R} \int_{\RR \backslash \{0\}}\left( \left( x(s-) + \kappa x(s-)z \right)^q - x^q(s-) \right) \widetilde N (ds,dz).
\end{equation*}
Now applying the similar ideas of estimating $J_{21}$ and $J_{22}$ in the proof of Theorem \ref{thm:possol}, we obtain the following estimates of $I_{21}$ and $I_{22}$
\begin{equation*}
I_{21} < \frac{q(q-1)(\kappa^2+2)C_{\alpha} }{2(2-\alpha)} \int_0^{t \wedge \eta_R} x^q (s-) ds
\end{equation*}
and
\begin{equation*}
I_{22}\leq \frac{2\kappa^{q}C_{\alpha}}{\alpha-q}\int_{0}^{t\wedge\eta_R}x^{q}(s-)ds.
\end{equation*}
Then,
\begin{equation*}
I_2<\left(\frac{q(q-1)(\kappa^2+2)C_{\alpha}}{2(2-\alpha)} +\frac{2\kappa^qC_{\alpha}}{\alpha - q}\right)\int_{0}^{t\wedge \eta_R}x^q(s-)ds.
\end{equation*}
Taking expectations on both sides of \eqref{qmomentthm:meq}, we have 
\begin{align*}
&~~~~\EE[x^q(t \wedge \eta_R)] \\
&< x^q_0+\EE\left[\int_{0}^{t\wedge \eta_R}\left(q\mu x^{q-1}(s)+\left(\frac{q(q-1)(\kappa^2+2)C_{\alpha}}{2(2-\alpha)}+\frac{2\kappa^qC_{\alpha}}{\alpha - q}-q\lambda\right)x^q(s)\right)ds\right]\\
&\leq x_0^q+q\mu T+\EE\left[\int_{0}^{t\wedge \eta_R}\left(q(\mu-\lambda)+\frac{q(q-1)(\kappa^2+2)C_{\alpha}}{2(2-\alpha)}+\frac{2\kappa^qC_{\alpha}}{\alpha - q}\right)x^q(s)ds\right]\\
&\leq C_1+C_2\EE\left[\int_{0}^{t\wedge \eta_R}x^q(s)ds\right]\\
&\leq C_1+C_2\int_{0}^{t}\EE\left[x^q(s\wedge \eta_R)\right]ds,
\end{align*}
where $C_1=x_0^q+q\mu T$, $C_2$ is a constant satisfies $C_2\geq q(\mu-\lambda)+ \frac{q(q-1)(\kappa^2+2)C_{\alpha}}{2(2-\alpha)}+\frac{2\kappa^qC_{\alpha}}{\alpha - q}$.
The Gronwall inequality shows
\begin{equation*}
\EE[x^q(t \wedge \eta_R)]\leq C_1 e^{C_2t}.
\end{equation*}
Finally, we complete the proof by the Fatou lemma. \eproof
\end{proof}

\section{Positivity preserving EM method}\label{sec4}
In this section, we first construct the positivity preserving EM method for the model \eqref{mainSDE}.
Let $\Delta\in(0,1)$ be the step size, $N=\lfloor T/\Delta\rfloor$. Grid points $t_k$ are taken as $t_k=k\Delta$, $k=0,1,2,\ldots, N-1$. We form the discrete-time positivity preserving EM solution by 
\begin{equation}\label{Num}
\left\{
\begin{array}{lr}
X_{\Delta}(t_{k+1})={X}_{\Delta}(t_{k})+\left(\mu-\lambda\widetilde{X}_{\Delta}(t_{k})\right)\Delta+\kappa\widetilde{X}_{\Delta}(t_{k})\Delta L^{\alpha}_k, \\
\widetilde{X}_{\Delta}(t_{k+1})=X_{\Delta}(t_{k+1})\vee \Delta,
\end{array}
\right.
\end{equation}
where $X_{\Delta}(0)=\widetilde{X}_{\Delta}(0)=x_0$, $\Delta L^{\alpha}_k=L^{\alpha}(t_{k+1})-L^{\alpha}(t_{k})$, and $\widetilde{X}_{\Delta}(t_k)$ is the approximate to $x(t_k)$.

We define $\eta(t)=t_k$, for any $t\in(t_k,t_{k+1}]$. Thus,
\begin{equation*}
\widetilde{X}_{\Delta}(\eta(t))={X}_{\Delta}(\eta(t))\vee \Delta.
\end{equation*}
The continuous time version of \eqref{Num} is given by
\begin{equation}\label{continuousPPEM}
\tilde{x}_{\Delta}(t)=x_0+\int_{0}^{t}\left(\mu-\lambda\widetilde{X}_{\Delta}(\eta(s))\right)ds+\int_{0}^{t}\kappa\widetilde{X}_{\Delta}(\eta(s-))dL^{\alpha}(s).
\end{equation}
Clearly,
$\tilde{x}_{\Delta}(t_k)=X_{\Delta}(t_k)$ for $k=0,1,2,\ldots ,N$.

Next, we present three lemmas to prove the strong convergence.

\begin{lemma}\label{lemma:2}
For $\alpha\in(1,2)$ and $q\in[1,\alpha)$, we have 
\begin{equation*}
\sup_{0\leq t\leq T}\EE\left[\left|\tilde{x}_{\Delta}(t)\right|^q\right]\leq C,
\end{equation*}
for some constant $C=C(q,T,\alpha,\mu,\lambda,\kappa,x_0)$. 
\end{lemma}
\begin{proof}
For any real number $R>x_0$, define the stopping time 
\begin{equation}\label{rhoR}
\rho_R=\inf\{t\geq 0:|\tilde{x}_{\Delta}(t)|\geq R\}.
\end{equation}
For any $q\in[1,\alpha)$ and any $t\in[0,T]$, by the elementary inequality and H\"older's inequality we derive from \eqref{continuousPPEM} that
\begin{align*}
\left|\tilde{x}_{\Delta}(t\wedge\rho_R)\right|^q&=\left|x_0+\int_{0}^{t\wedge\rho_R}\left(\mu-\lambda\widetilde{X}_{\Delta}(\eta(s))\right)ds+\int_{0}^{t\wedge\rho_R}\kappa\widetilde{X}_{\Delta}(\eta(s-))dL^{\alpha}(s)\right|^q\\
&\leq 3^{q-1}|x_0|^q+3^{q-1}\left|\int_{0}^{t\wedge\rho_R}\left(\mu-\lambda \widetilde{X}_{\Delta}(\eta(s))\right)ds\right|^q\\
&\quad +3^{q-1}\left|\int_{0}^{t\wedge\rho_R}\kappa\widetilde{X}_{\Delta}(\eta(s-))dL^{\alpha}(s)\right|^q\\
&\leq3^{q-1}|x_0|^q+3^{q-1}T^{q-1}\int_{0}^{t\wedge\rho_R}\left|\mu-\lambda\widetilde{X}_{\Delta}(\eta(s))\right|^qds\\
&\quad +3^{q-1}\left|\int_{0}^{t\wedge\rho_R}\kappa\widetilde{X}_{\Delta}(\eta(s-))dL^{\alpha}(s)\right|^q\\
&\leq 3^{q-1}|x_0|^q+6^{q-1}\mu^qT^q+6^{q-1}\lambda^qT^{q-1}\int_{0}^{t}\left|\widetilde{X}_{\Delta}\left(\eta(s)\wedge\rho_R\right)\right|^q ds\\
&\quad +3^{q-1}\left|\int_{0}^{t}\kappa\widetilde{X}_{\Delta}\left(\eta(s-)\wedge\rho_R-\right)dL^{\alpha}(s)\right|^q.
\end{align*}
Taking expectations both sides, by Lemma \ref{L_alpha_moment_2} and the H\"older inequality, we obtain
\begin{align*}
\EE&\left[\left|\tilde{x}_{\Delta}(t\wedge\rho_R)\right|^q\right]\\
&\leq 3^{q-1}|x_0|^q+6^{q-1}\mu^qT^q+6^{q-1}\lambda^qT^{q-1}\int_{0}^{t}\EE\left[\left|\widetilde{X}_{\Delta}\left(\eta(s)\wedge\rho_R\right)\right|^q \right]ds\\
&\quad +3^{q-1}\kappa^qC_{q,\alpha}
\left(\int_{0}^{t}\EE\left[\left|\widetilde{X}_{\Delta}\left(\eta(s-)\wedge\rho_R-\right)\right|^{\alpha}\right]ds\right)^{\frac{q}{\alpha}}\\
&\leq 3^{q-1}|x_0|^q+6^{q-1}\mu^qT^q+6^{q-1}\lambda^qT^{q-\frac{q}{\alpha}}\left(\int_{0}^{t}\EE\left[\left|\widetilde{X}_{\Delta}(\eta(s)\wedge\rho_R)\right|^{\alpha}\right]ds\right)^{\frac{q}{\alpha}}\\
&\quad +3^{q-1}\kappa^qC_{q,\alpha}
\left(\int_{0}^{t}\EE\left[\left|\widetilde{X}_{\Delta}\left(\eta(s-)\wedge\rho_R-\right)\right|^{\alpha}\right]ds\right)^{\frac{q}{\alpha}}\\
&\leq3^{q-1}|x_0|^q+6^{q-1}\mu^qT^q\\
&\quad +\left(6^{q-1}\lambda^qT^{q-\frac{q}{\alpha}}+3^{q-1}\kappa^qC_{q,\alpha}\right)\left(\int_{0}^{t}\Big(\sup_{0\leq \varrho\leq s}\EE\left[\left|\tilde{x}_{\Delta}(\varrho\wedge\rho_R)\right|^{\alpha}\right]+\Delta^{\alpha}\Big)ds\right)^{\frac{q}{\alpha}}\\
&\leq 3^{q-1}|x_0|^q+6^{q-1}\mu^qT^q+6^{q-1}\lambda^qT^{2q}+3^{q-1}\kappa^qC_{q,\alpha}T^{\frac{q(\alpha+1)}{\alpha}}\\
&\quad +\left(6^{q-1}\lambda^qT^{q-\frac{q}{\alpha}}+3^{q-1}\kappa^qC_{q,\alpha}\right)\left(\int_{0}^{t}\sup_{0\leq \varrho\leq s}\EE\left[\left|\tilde{x}_{\Delta}(\varrho\wedge\rho_R)\right|^{\alpha}\right]ds\right)^{\frac{q}{\alpha}}.
\end{align*}
Since the inequality above holds for any $t\in[0,T]$, we can see
\begin{equation*}
\sup_{0\leq \iota \leq t}\EE\left[\left|\tilde{x}_{\Delta}(\iota\wedge\rho_R)\right|^q\right]\leq  A+B\left(\int_{0}^{t}\sup_{0\leq \varrho\leq s}\EE\left[\left|\tilde{x}_{\Delta}(\varrho\wedge\rho_R)\right|^{\alpha}\right]ds\right)^{\frac{q}{\alpha}},
\end{equation*}
where 
\begin{equation*}
A=3^{q-1}|x_0|^q+6^{q-1}\mu^qT^q+6^{q-1}\lambda^qT^{2q}+3^{q-1}\kappa^qC_{q,\alpha}T^{\frac{q(\alpha+1)}{\alpha}}
\end{equation*}
and
\begin{equation*}
B=6^{q-1}\lambda^qT^{q-\frac{q}{\alpha}}+3^{q-1}\kappa^qC_{q,\alpha}.
\end{equation*}
Then, applying the non-linear form of the Gronwall inequality \cite{WillettWong1965} yields
\begin{equation*}
\sup_{0\leq \iota \leq t}\EE\left[\left|\tilde{x}_{\Delta}(\iota\wedge\rho_R)\right|^q\right]\leq A+B\frac{\left(\int_{0}^{t}e(s)A^{\alpha}ds\right)^{\frac{q}{\alpha}}}{1-[1-e(t)]^{\frac{q}{\alpha}}},
\end{equation*}
where 
\begin{equation*}
e(t)=\exp\left(-\int_{0}^{t}B^{\alpha}ds\right).
\end{equation*}
Since $\rho_R\rightarrow\infty$ as $R\rightarrow\infty$, setting $t=T$ and $R\rightarrow\infty$, we can obtain 
\begin{equation*}
\sup_{0\leq t\leq T}\EE\left[|\tilde{x}_{\Delta}(t)|^q\right]\leq C,
\end{equation*}
where $C$ is a constant independent of $\Delta$, which completes the proof. \eproof
\end{proof}

The following lemma indicates that $\tilde{x}_{\Delta}(t)$ and $X_{\Delta}(\eta(t))$ are close to each other.
\begin{lemma}\label{lemma:3}
For any $\alpha\in(1,2)$ and $q\in[1,\alpha)$, we have
\begin{equation*}
\sup_{0\leq t\leq T}\EE\left[\left|\tilde{x}_{\Delta}(t)-{X}_{\Delta}(\eta(t))\right|^q\right]\leq C\Delta^{q/\alpha}.
\end{equation*}
\end{lemma}
\begin{proof}
It follows from \eqref{Num} and \eqref{continuousPPEM} that
\begin{equation*}
\tilde{x}_{\Delta}(t)=X_{\Delta}(t_k)+\int_{t_k}^{t}(\mu-\lambda\widetilde{X}_{\Delta}(\eta(s)))ds+\int_{t_k}^{t}\kappa\widetilde{X}_{\Delta}(\eta(s-))dL^{\alpha}(s).
\end{equation*}
Taking the absolute value, the expectation and using independence increments property, for every $k=0,1,2,\ldots, N-1$, we obtain for $t\in(t_k,t_{k+1}]$,
\begin{align*}
\EE&\left[\left|\tilde{x}_{\Delta}(t)-{X}_{\Delta}(\eta(t))\right|^q\right] \\
&= \EE\left[\left|\int_{t_k}^{t}(\mu-\lambda\widetilde{X}_{\Delta}(\eta(s)))ds+\int_{t_k}^{t}\kappa\widetilde{X}_{\Delta}(\eta(s-))dL^{\alpha}(s)\right|^q\right]\\
&\leq 2^{q-1}\Delta^{q-1}\EE\left[\int_{t_k}^{t}\left|\mu-\lambda\widetilde{X}_{\Delta}(\eta(s))\right|^q ds\right]+2^{q-1}\EE\left[\left|\int_{t_k}^{t}\kappa\widetilde{X}_{\Delta}(\eta(s-))d L^{\alpha}(s)\right|^q\right]\\
&\leq 4^{q-1}\Delta^{q}\mu^q+4^{q-1}\Delta^{q-1}\lambda^q\EE\left[\int_{t_k}^{t}\left|\widetilde{X}_{\Delta}(\eta(s))\right|^qds\right]+2^{q-1}\kappa^q\Delta^{\frac{q}{\alpha}}\EE\left[\left|\widetilde{X}_{\Delta}(t_k)\right|^q\right]\\
&=4^{q-1}\Delta^{q}\mu^q+\left(4^{q-1}\Delta^{q}\lambda^q+2^{q-1}\kappa^q\Delta^{\frac{q}{\alpha}}\right)\EE\left[\left|\widetilde{X}_{\Delta}(t_k)\right|^q\right]\\
&\leq 4^{q-1}\Delta^q(\mu^q+\Delta^q\lambda^q)+2^{q-1}\kappa^q\Delta^{q+\frac{q}{\alpha}}+\left(4^{q-1}\Delta^{q}\lambda^q+2^{q-1}\kappa^q\Delta^{\frac{q}{\alpha}}\right)\EE\left[\left|{X}_{\Delta}(t_k)\right|^q\right]\\
&\leq 4^{q-1}\Delta^q(\mu^q+\Delta^q\lambda^q)+2^{q-1}\kappa^q\Delta^{q+\frac{q}{\alpha}}+C\left(4^{q-1}\Delta^{q}\lambda^q+2^{q-1}\kappa^q\Delta^{\frac{q}{\alpha}}\right)\\
&\leq C\Delta^{q/\alpha}, 
\end{align*}
where we have applied Lemma \ref{lemma:2} in the penultimate line. \eproof
\end{proof}	

\begin{lemma}\label{xxerror}
Suppose that Assumptions \ref{assp:coef} and \ref{assp:levy} hold. For any $\Delta\in\left(0,\frac{\mu-1}{\lambda}\wedge\frac{1}{\lambda}\right)$ and $\alpha\in(1,2)$, we have 
\begin{equation*}
\PP\left(X_{\Delta}(t_{k+1})<\Delta\right)=0, ~~\text{as}~ \Delta\to 0,
\end{equation*}
where $k=0,1,2,\ldots$ is the number of steps. 
\end{lemma}
\begin{proof}
Recall \eqref{Num}, if $X_{\Delta}(t_k)<\Delta$, combine Assumption \ref{assp:levy}, we can arrive at
\begin{align*}
X_{\Delta}(t_{k+1})&=X_{\Delta}(t_k)+\left(\mu-\lambda\widetilde{X}_{\Delta}(t_k)\right)\Delta+\kappa\widetilde{X}_{\Delta}(t_k)\Delta L^{\alpha}_k\\
&=X_{\Delta}(t_k)+(\mu-\lambda\Delta)\Delta+\kappa\Delta \Delta L^{\alpha}_k\\
&=X_{\Delta}(t_k)+\mu\Delta-\lambda\Delta^2+\kappa\Delta L^{\alpha}_k\Delta\\
&>X_{\Delta}(t_k)+(\mu-1)\Delta-\lambda\Delta^2.
\end{align*}
By Assumption \ref{assp:coef} and $\Delta\in\left(0,\frac{\mu-1}{\lambda}\wedge\frac{1}{\lambda}\right)$, it can be inferred that
\begin{equation*}
(\mu-1)\Delta-\lambda\Delta^2>0. 
\end{equation*}
This indicates that if there exists a integer $k$ such that $X_{\Delta}(t_k)<\Delta$, then $X_{\Delta}(t_{k+1})>X_{\Delta}(t_k)$ holds.  Below we prove the case where $X_{\Delta}(t_k)>\Delta$, which requires proof
\begin{align*}
\PP&\left(X_{\Delta}(t_k)+\left(\mu-\lambda X_{\Delta}(t_k)\right)\Delta+\kappa X_{\Delta}(t_k)\Delta L_k^{\alpha}<\Delta\right)\\
&=\PP\left(\kappa X_{\Delta}(t_k)\Delta L_k^{\alpha}<\Delta-X_{\Delta}(t_k)-(\mu-\lambda X_{\Delta}(t_k))\Delta\right)\\
&=\PP\left(\Delta L_{k}^{\alpha}<\frac{(1-\mu)\Delta-(1-\lambda\Delta)X_{\Delta}(t_k)}{\kappa X_{\Delta}(t_k)}\right).
\end{align*}
The random variable $\Delta L_k^{\alpha}$ represents the increment of an $\alpha$-stable process, and its tail behavior is controlled by a power-law decay. Specifically, as $ x \to -\infty $, the tail probability behaves like,
\begin{equation*}
P(\Delta L_k^{\alpha} < x) \sim |x|^{-\alpha}.
\end{equation*}
Then, we can get
\begin{equation*}
\PP\left(\Delta L_{k}^{\alpha}<\frac{(1-\mu)\Delta-(1-\lambda\Delta)X_{\Delta}(t_k)}{\kappa X_{\Delta}(t_k)}\right)\sim \left(\frac{|(1-\mu)\Delta-(1-\lambda\Delta)X_{\Delta}(t_k)|}{\kappa X_{\Delta}(t_k)}\right)^{-\alpha}. 
\end{equation*}
It is obvious that when $\Delta\to 0$, we have
\begin{equation*}
\left(\frac{|(1-\mu)\Delta-(1-\lambda\Delta)X_{\Delta}(t_k)|}{\kappa X_{\Delta}(t_k)}\right)^{-\alpha}\to \left(\frac{1}{\kappa}\right)^{-\alpha}. 
\end{equation*}
Recall $\kappa\in(0,1)$ and Assumption \ref{assp:levy}, we can derive 
\begin{equation*}
\PP\left(\Delta L_{k}^{\alpha}<\frac{(1-\mu)\Delta-(1-\lambda\Delta)X_{\Delta}(t_k)}{\kappa X_{\Delta}(t_k)}\right)\rightarrow\PP\left(\Delta L_k^{\alpha}<-\frac{1}{\kappa}\right)\rightarrow0, ~~\text{as}~\Delta\to 0.
\end{equation*}
This completes the proof.	\eproof
\end{proof}

\begin{theorem}\label{thm:conver}
For any $\Delta\in\left(0,\frac{\mu-1}{\lambda}\wedge\frac{1}{\lambda}\right)$ and $q\in[1,\alpha)$, we have
\begin{equation}\label{Thmequ}
\sup_{0\leq t\leq T}\EE\left[\left|x(t)-\tilde{x}_{\Delta}(t)\right|^q\right]\leq C\Delta^{q/\alpha}.
\end{equation}
\end{theorem}
\begin{proof}
By \eqref{mainSDE} gives
\begin{equation*}
x(t)=x_0+\int_{0}^{t}\left(\mu-\lambda x(s)\right)ds+\int_{0}^{t}\kappa x(s-)dL^{\alpha}(s),
\end{equation*}
combining \eqref{continuousPPEM} we have
\begin{equation*}
x(t)-\tilde{x}_{\Delta}(t)=\int_{0}^{t}-\lambda\left(x(s)-\widetilde{X}_{\Delta}(\eta(s))\right)ds+\int_{0}^{t}\kappa\left(x(s-)-\widetilde{X}_{\Delta}(\eta(s-))\right)dL^{\alpha}(s).
\end{equation*}
Recall $\tau_R$ defined in \eqref{tauR} and $\rho_R$ defined in \eqref{rhoR}. Define $\zeta_R=\tau_R\wedge\rho_R$ and $e(t)=x(t)-\tilde{x}_{\Delta}(t)$, we get
\begin{align*}
e(t\wedge\zeta_R)&=\int_{0}^{t\wedge\zeta_R}-\lambda\left(x(s)-\widetilde{X}_{\Delta}(\eta(s))\right)ds\\
&\quad +\int_{0}^{t\wedge\zeta_R}\kappa\left(x(s-)-\widetilde{X}_{\Delta}(\eta(s-))\right)dL^{\alpha}(s).
\end{align*}

Then we can obtain
\begin{align}\label{xXerror}
\EE\left[\left|e(t\wedge\zeta_R)\right|^q\right] &=\EE \bigg[\bigg|\int_{0}^{t\wedge\zeta_R}-\lambda\left(x(s)-\widetilde{X}_{\Delta}(\eta(s))\right)ds\nonumber\\
&\quad +\int_{0}^{t\wedge\zeta_R}\kappa \left(x(s-)-\widetilde{X}_{\Delta}(\eta(s-))\right)dL^{\alpha}(s)\bigg|^q \bigg]\nonumber \\
&\leq 2^{q-1}\EE\left[ \left|\int_{0}^{t\wedge\zeta_R}-\lambda\left(x(s)-\widetilde{X}_{\Delta}(\eta(s))\right)ds\right|^q \right] \nonumber \\
&\quad +2^{q-1}\EE\left[\left|\int_{0}^{t\wedge\zeta_R}\kappa \left(x(s-)-\widetilde{X}_{\Delta}(\eta(s-))\right)dL^{\alpha}(s)\right|^q\right] \nonumber \\
&=\Theta_1+\Theta_2.
\end{align}
By the H\"older inequality, we obtain
\begin{align*}
\Theta_1&=2^{q-1}\EE\left[\left|\int_{0}^{t\wedge\zeta_R}-\lambda\left(x(s)-\widetilde{X}_{\Delta}(\eta(s))\right)ds\right|^q\right]  \nonumber\\
&\leq 2^{q-1}\lambda^qT^{q-1}\int_{0}^{t\wedge\zeta_R}\EE\left[\left|x(s)-\widetilde{X}_{\Delta}(\eta(s))\right|^q\right]ds. 
\end{align*}
Applying the triangle inequality gives
\begin{align*}
\EE\left[\left|x(s)-\widetilde{X}_{\Delta}(\eta(s))\right|^{q}\right]&\leq3^{q-1} \Big(\EE\left[\left|x(s)-\tilde{x}_{\Delta}(s)\right|^{q}\right]+\EE\left[\left|\tilde{x}_{\Delta}(s)-{X}_{\Delta}(\eta(s))\right|^{q}\right]\\
&\quad +\EE\left[\left|X_{\Delta}(\eta(s))-\widetilde{X}_{\Delta}(\eta(s))\right|^{q}\right]\Big).
\end{align*}
It follows from Lemma \ref{lemma:3} that 
\begin{equation*}
\EE\left[\left|\tilde{x}_{\Delta}(s)-{X}_{\Delta}(\eta(s))\right|^{q}\right]\leq C\Delta^{q/\alpha}.
\end{equation*}
From Lemma \ref{xxerror}, we can  find a samll enough $\delta=\delta(\varepsilon)\in(0,1)$ such that 
\begin{equation*}
\PP(\Omega_1)\geq 1-\varepsilon, ~~\forall~ \Delta\in\left(0,\frac{\mu-1}{\lambda}\wedge\frac{1}{\lambda}\right), 
\end{equation*}
where 
\begin{equation*}
\Omega_1=\{|X_{\Delta}(\eta(t))-\widetilde{X}_{\Delta}(\eta(t))|<{\delta}\}.
\end{equation*}
Hence, 
\begin{align*}
\EE\left[\left|X_{\Delta}(\eta(s))-\widetilde{X}_{\Delta}(\eta(s))\right|^{q}\right]&=\EE\left[\left|X_{\Delta}(\eta(s))-\widetilde{X}_{\Delta}(\eta(s))\right|^{q}\mathbb{I}_{\Omega_1}\right]\\
&\quad +\EE\left[\left|X_{\Delta}(\eta(s))-\widetilde{X}_{\Delta}(\eta(s))\right|^{q}\mathbb{I}_{\Omega_1^c}\right]\\
&\leq \delta^{q}+\PP(\Omega_1^c)\EE\left[\left|X_{\Delta}(\eta(s))-\widetilde{X}_{\Delta}(\eta(s))\right|^{q}\right]\\
&\leq \delta^{q}+2^{q-1}\varepsilon\left\{\EE\left[\left|X_{\Delta}(\eta(s))\right|^{q}\right]+\EE\left[\left|\widetilde{X}_{\Delta}(\eta(s))\right|^{q}\right]\right\}\\
&\leq \delta^{q}+2^{q-1}C\varepsilon, 
\end{align*}
where Lemma \ref{lemma:2} is used in the final step. Therefore, 
\begin{align}\label{Theta1}
\Theta_1&\leq 6^{q-1}\lambda^qT^{q}C\Delta^{q/\alpha}+6^{q-1}\lambda^qT^{q}(\delta^q+2^{q-1}C\epsilon) \nonumber\\
&\quad +6^{q-1}\lambda^qT^{q-1}\int_{0}^{t}\EE\left[|e(s\wedge\zeta_R)|^q\right]ds.
\end{align}

For the estimate of $\Theta_2$, using Lemma \ref{L_alpha_moment_2} yields
\begin{align}\label{Theta2}
\Theta_2&=2^{q-1}\EE\left[\left|\int_{0}^{t\wedge\zeta_R}\kappa \left(x(s-)-\widetilde{X}_{\Delta}(\eta(s-))\right)dL^{\alpha}(s)\right|^q \right]\nonumber\\
&\leq 2^{q-1}\kappa^qC_{q,\alpha}\left(\int_{0}^{t}\EE\left[\left|x(s\wedge\zeta_R)-\widetilde{X}_{\Delta}\left(\eta(s)\wedge\zeta_R\right)\right|^{\alpha} \right]ds\right)^{\frac{q}{\alpha}} \nonumber\\
&\leq \left(2^{2q-1}+2^{3q-1-\frac{q}{\alpha}}\right)\kappa^qC_{q,\alpha}3^{q-\frac{q}{\alpha}}C^{\frac{q}{\alpha}}T^{q/\alpha}+2^{q-1}\kappa^qC_{q,\alpha}6^{q-\frac{q}{\alpha}}T^{q/\alpha}\Delta^q \nonumber\\
&\quad+ 2^{q-1}\kappa^qC_{q,\alpha}3^{q-\frac{q}{\alpha}}\left(\int_{0}^{t}\EE\left[|e(s\wedge\zeta_R)|^{\alpha}\right]ds\right)^{\frac{q}{\alpha}}.
\end{align}
Substituting \eqref{Theta1} and \eqref{Theta2} into \eqref{xXerror}, and using the H\"older inequality, we obtain
\begin{align*}
\EE\left[\left|e(t\wedge\zeta_R)\right|^q\right]&\leq 6^{q-1}\lambda^qT^{q}C\Delta^{q/\alpha}+6^{q-1}\lambda^qT^{q}(\delta^q+2^{q-1}C\epsilon) \\
&\quad +6^{q-1}\lambda^qT^{q-1}\int_{0}^{t}\EE\left[|e(s\wedge\zeta_R)|^q\right]ds +2^{q-1}\kappa^qC_{q,\alpha}6^{q-\frac{q}{\alpha}}T^{q/\alpha}\Delta^q\\
&\quad + \left(2^{2q-1}+2^{3q-1-\frac{q}{\alpha}}\right)\kappa^qC_{q,\alpha}3^{q-\frac{q}{\alpha}}C^{\frac{q}{\alpha}}T^{q/\alpha} \\
&\quad+ 2^{q-1}\kappa^qC_{q,\alpha}3^{q-\frac{q}{\alpha}}\left(\int_{0}^{t}\EE\left[|e(s\wedge\zeta_R)|^{\alpha}\right]ds\right)^{\frac{q}{\alpha}}\\
&\leq 6^{q-1}\lambda^qT^{q}C\Delta^{q/\alpha}+6^{q-1}\lambda^qT^{q}(\delta^q+2^{q-1}C\epsilon)\\
&\quad +\left(2^{2q-1}+2^{3q-1-\frac{q}{\alpha}}\right)\kappa^qC_{q,\alpha}3^{q-\frac{q}{\alpha}}C^{\frac{q}{\alpha}}T^{q/\alpha}+2^{q-1}\kappa^qC_{q,\alpha}6^{q-\frac{q}{\alpha}}T^{q/\alpha}\Delta^q\\
&\quad +\left(6^{q-1}\lambda^qT^{q-\frac{q}{\alpha}}+2^{q-1}\kappa^qC_{q,\alpha}3^{q-\frac{q}{\alpha}}\right)\left(\int_{0}^{t}\EE\left[|e(s\wedge\zeta_R)|^{\alpha}\right]ds\right)^{\frac{q}{\alpha}}\\
&\leq A_1+B_1\left(\int_{0}^{t}\EE\left[|e(s\wedge\zeta_R)|^{\alpha}\right]ds\right)^{\frac{q}{\alpha}}, 
\end{align*}
where 
\begin{align*}
A_1&=6^{q-1}\lambda^qT^{q}C\Delta^{q/\alpha}+6^{q-1}\lambda^qT^{q}(\delta^q+2^{q-1}C\epsilon)\\
&\quad +\left(2^{2q-1}+2^{3q-1-\frac{q}{\alpha}}\right)\kappa^qC_{q,\alpha}3^{q-\frac{q}{\alpha}}C^{\frac{q}{\alpha}}T^{q/\alpha}+2^{q-1}\kappa^qC_{q,\alpha}6^{q-\frac{q}{\alpha}}T^{q/\alpha}\Delta^q
\end{align*}
and
\begin{equation*}
B_1=6^{q-1}\lambda^qT^{q-\frac{q}{\alpha}}+2^{q-1}\kappa^qC_{q,\alpha}3^{q-\frac{q}{\alpha}}.
\end{equation*}
This immediately gives
\begin{equation*}
\sup_{0\leq r \leq t}\EE\left[|e(r\wedge\zeta_R)|^q\right]\leq A_1+B_1\left(\int_{0}^{t}\sup_{0\leq r \leq s}\EE\left[|e(r\wedge\zeta_R)|^{\alpha}\right]ds\right)^{\frac{q}{\alpha}}.
\end{equation*}

Finally, by the non-linear form of the Gronwall inequality \cite{WillettWong1965}, it can be concluded that
\begin{equation*}
\sup_{0\leq r \leq t}\EE\left[|e(r\wedge\zeta_R)|^q\right]\leq A_1+B_1\frac{\left(\int_{0}^{T}e_1(s)A_1^{\alpha}ds\right)^{\frac{q}{\alpha}}}{1-[1-e_1(t)]^{\frac{q}{\alpha}}},
\end{equation*}
where 
\begin{equation*}
e_1(t)=\exp\left(-\int_{0}^{T}\left[6^{q-1}\lambda^qT^{q-\frac{q}{\alpha}}+2^{q-1}\kappa^qC_{q,\alpha}3^{q-\frac{q}{\alpha}}\right]^{\alpha} ds \right)=\exp\left(-\int_{0}^{T}B_1^{\alpha}ds\right).
\end{equation*}
Since $\delta$ is small enough and $\varepsilon\in(0,1)$ is arbitrary, letting $t=T$, $R\rightarrow\infty$ and applying Fatou's lemma, we can obtain the required inequality \eqref{Thmequ} and complete the proof. \eproof
\end{proof}

\begin{theorem}\label{thm:conver2}
For any $\Delta\in\left(0,\frac{\mu-1}{\lambda}\wedge\frac{1}{\lambda}\right)$ and $q\in[1,\alpha)$, we have
\begin{equation*}
\sup_{0\leq t\leq T}\EE\left[\left|x(t)-\widetilde{X}_{\Delta}(\eta(t))\right|^q\right]\leq C\Delta^{q/\alpha}.
\end{equation*}
\end{theorem}
\begin{proof}
Due to
\begin{align*}
\left|x(t)-\widetilde{X}_{\Delta}(\eta(t))\right|^q&\leq 3^{q-1}\Big(|x(t)-\tilde{x}_{\Delta}(t)|^q+\left|\tilde{x}_{\Delta}(t)-X_{\Delta}(\eta(t))\right|^q\\
&\quad +\left|X_{\Delta}(\eta(t))-\widetilde{X}_{\Delta}(\eta(t))\right|^q\Big).
\end{align*}
Applying Lemmas \ref{lemma:3}, \ref{xxerror} and Theorem \ref{thm:conver},we obtain the assertion.    \eproof
\end{proof}

\section{Numerical simulations}\label{sec5}
\begin{expl}\label{Example:1}
We consider the following linear SDE
\begin{equation}\label{exsde:1}
dx(t) = (\mu - \lambda x(t))dt + \kappa x(t)dL^{\alpha}(t),  ~~x(0)=1.
\end{equation}
\end{expl}
Let $\mu = 1.5, \lambda = 2, \kappa = 0.5$. We can verify that they satisfy Assumption \ref{assp:coef}. Figure \ref{xt} shows the three paths generated by the positivity preserving EM numerical method.

In order to observe whether the parameters have an impact on the error, we acknowledge the following parameters in subsequent experiments \eqref{exsde:1}:
\begin{itemize}
\item $\mu=1.5$, $\lambda=2$, $\kappa=0.5$;
\item $\mu=2$, $\lambda=3$, $\kappa=0.5$;
\item $\mu=2$, $\lambda=3$, $\kappa=0.2$.
\end{itemize}

\begin{figure}[H]
\centering
\includegraphics[width=0.75\linewidth]{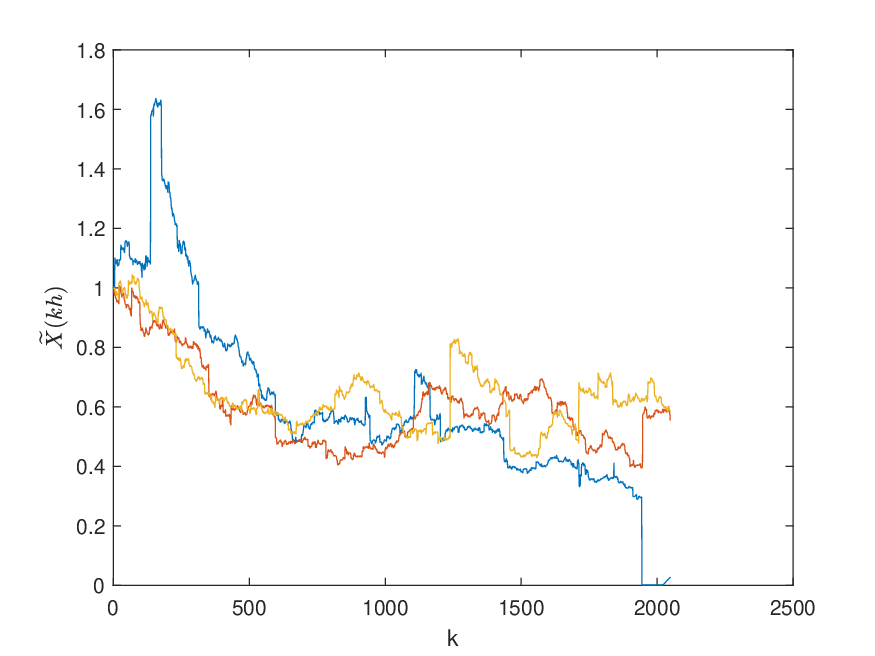}
\caption{Three paths generated by the positivity preserving EM numerical method}
\label{xt}
\end{figure}

We run $M=1000$ independent trajectories for every different stepsizes, $2^{-10}$, $2^{-11}$, $2^{-12}$, $2^{-13}$, $2^{-14}$, $2^{-16}$. The numerical solution with the stepsizes $2^{-16}$ is regarded as the exact solution. 
Tables \ref{table1}--\ref{table3} present the computational errors for different parameter choices. The numerical results indicate that as $\alpha$ decreases, the error reduces accordingly.

To clearly display the convergence rates, Figures \ref{ex1}--\ref{ex13} present the achieved errors versus stepsizes on a logarithmic scale. As predicted, the slopes of the errors (solid lines) and the reference dashed line match well, which indicates that the proposed scheme shows a strong convergence rate of order $1/\alpha$. Moreover, it is intuitively evident that variations in parameters do not influence the convergence rate.

\begin{table}[htph]
\centering
\renewcommand{\arraystretch}{1.2}
\setlength{\tabcolsep}{20pt}
\caption{Numerical results for \eqref{exsde:1} with $\mu=1.5$, $\lambda=2$, $\kappa=0.5$. }
\label{table1}
\begin{tabular}{ccccc}
\Xhline{2\arrayrulewidth}
$\Delta $ & $\alpha=1.8$   &  $\alpha=1.6$ &  $\alpha=1.4$  &  $\alpha=1.1$ \\ \Xhline{2\arrayrulewidth}
$2^{-14}$ &  $0.0007$  &  $0.0005$ & $0.0004$ & $0.0001$ \\
$2^{-13}$ &  $0.0011$   &  $0.0008$ & $0.0007$ & $0.0002$ \\
$2^{-12}$ &  $0.0016$   &  $0.0012$ & $0.0012$  & $0.0004$ \\
$2^{-11}$ &  $0.0023$   &  $0.0019$ & $0.0021$  & $0.0006$ \\
$2^{-10}$ &  $0.0039$  &  $0.0029$ & $0.0030$ & $0.0013$  \\ \Xhline{2\arrayrulewidth}
\end{tabular}
\end{table}

\begin{figure}[H]
\centering
\subfigure[{\label{fig:1}} $\alpha=1.8$]
{\includegraphics[width=0.49\linewidth]{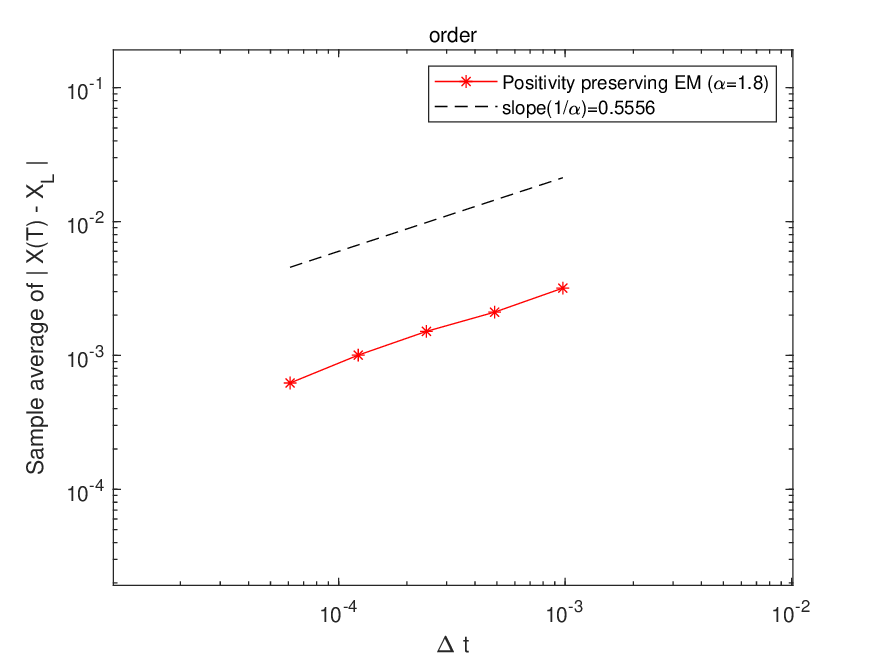}}
\subfigure[{\label{fig:2}} $\alpha=1.6$]
{\includegraphics[width=0.49\linewidth]{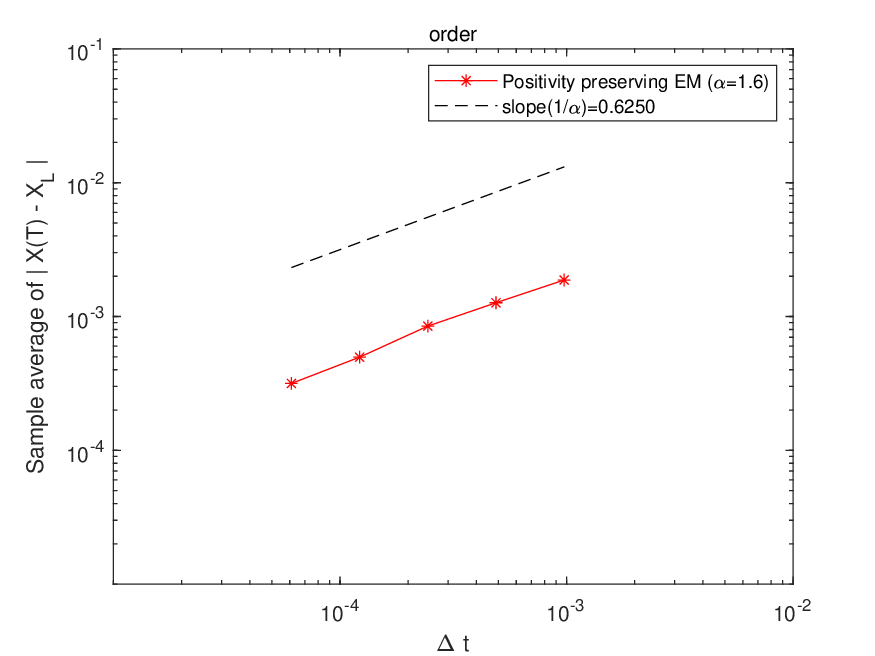}}
\subfigure[{\label{fig:3}} $\alpha=1.4$]
{\includegraphics[width=0.49\linewidth]{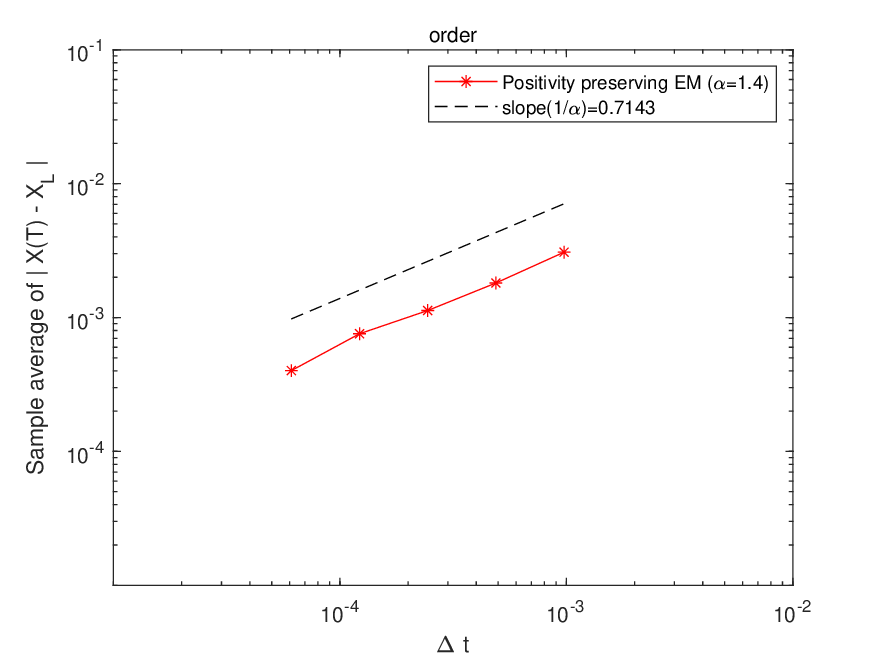}}
\subfigure[{\label{fig:4}} $\alpha=1.1$]
{\includegraphics[width=0.49\linewidth]{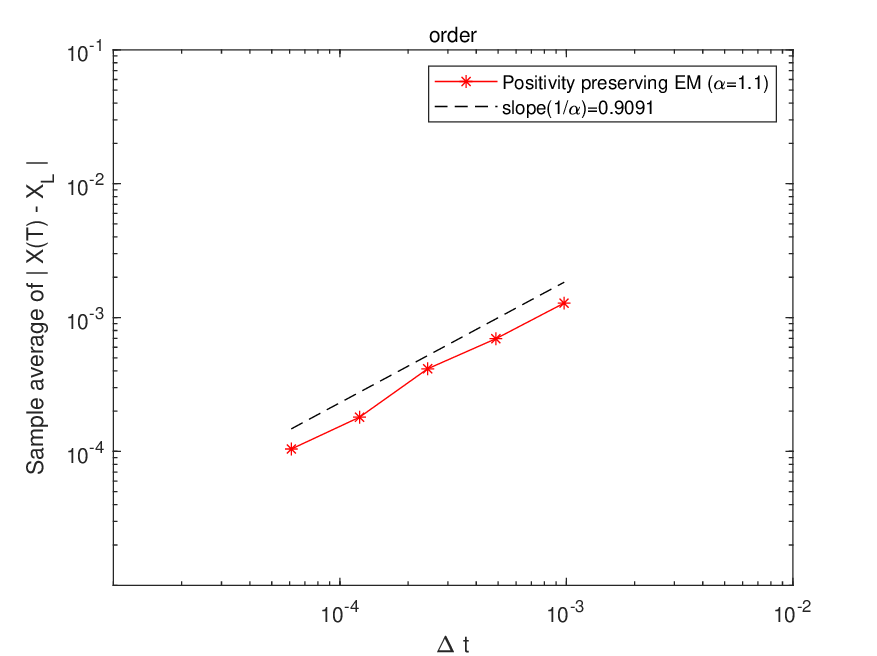}}
\caption{Errors versus stepsize $\Delta $ on log-log scale with $\mu=1.5$, $\lambda=2$, $\kappa=0.5$}
\label{ex1}
\end{figure}

\begin{table}[htph]
	\centering
	\renewcommand{\arraystretch}{1.2}
	\setlength{\tabcolsep}{20pt}
	\caption{Numerical results for \eqref{exsde:1} with $\mu=2$, $\lambda=3$, $\kappa=0.5$. }
	\label{table2}
	\begin{tabular}{ccccc}
		\Xhline{2\arrayrulewidth}
		$\Delta $ & $\alpha=1.8$   &  $\alpha=1.6$ &  $\alpha=1.4$  &  $\alpha=1.1$ \\ \Xhline{2\arrayrulewidth}
		$2^{-14}$ &  $0.0004$  &  $0.0003$ & $0.0002$ & $0.0551\times10^{-3}$ \\
		$2^{-13}$ &  $0.0006$  &  $0.0006$ & $0.0005$ & $0.1383\times10^{-3}$ \\
		$2^{-12}$ &  $0.0009$  &  $0.0009$ & $0.0007$ & $0.2191\times10^{-3}$ \\
		$2^{-11}$ &  $0.0012$   &  $0.0012$ &  $0.0011$  & $0.3997\times10^{-3}$ \\
		$2^{-10}$ &  $0.0019$  &  $0.0020$ &  $0.0019$ & $0.6640\times10^{-3}$  \\ \Xhline{2\arrayrulewidth}
	\end{tabular}
\end{table}

\begin{figure}[H]
	\centering
	\subfigure[{\label{fig:1_2}} $\alpha=1.8$]
	{\includegraphics[width=0.49\linewidth]{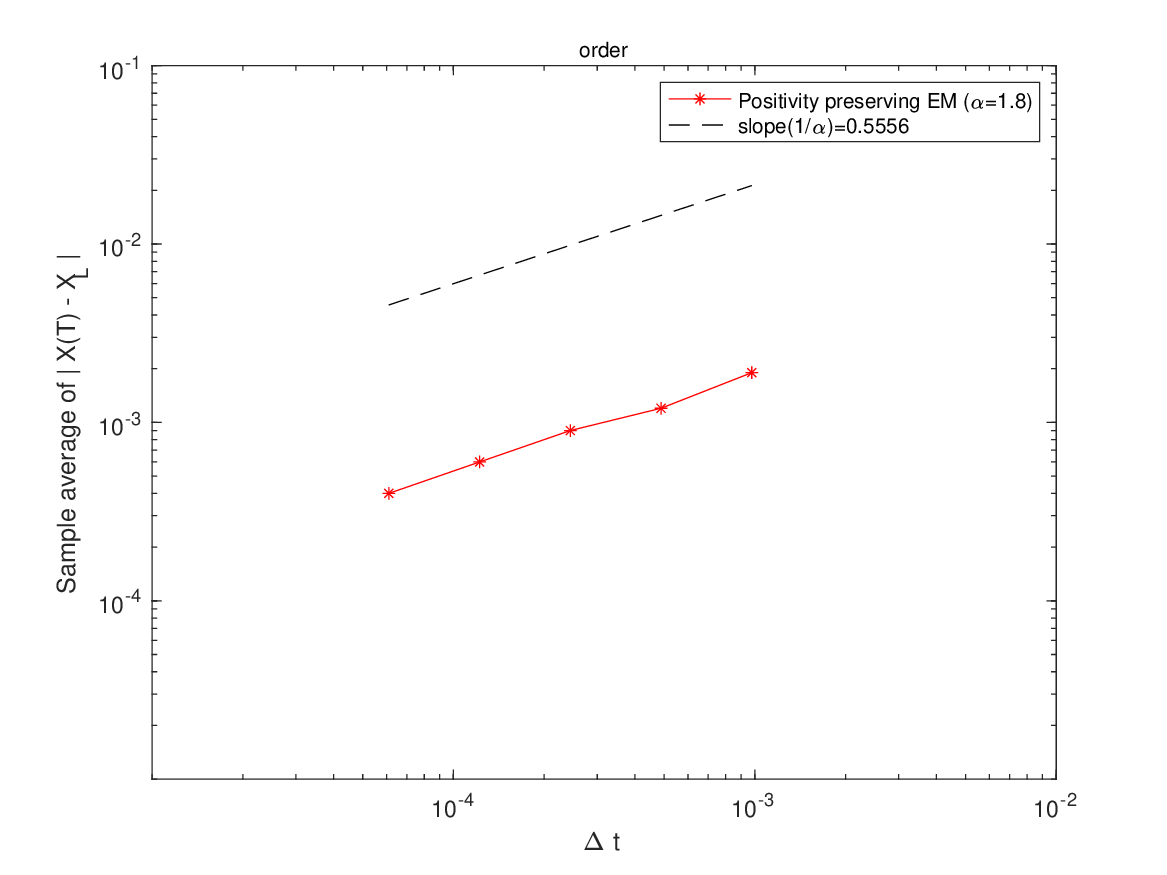}}
	\subfigure[{\label{fig:2_2}} $\alpha=1.6$]
	{\includegraphics[width=0.49\linewidth]{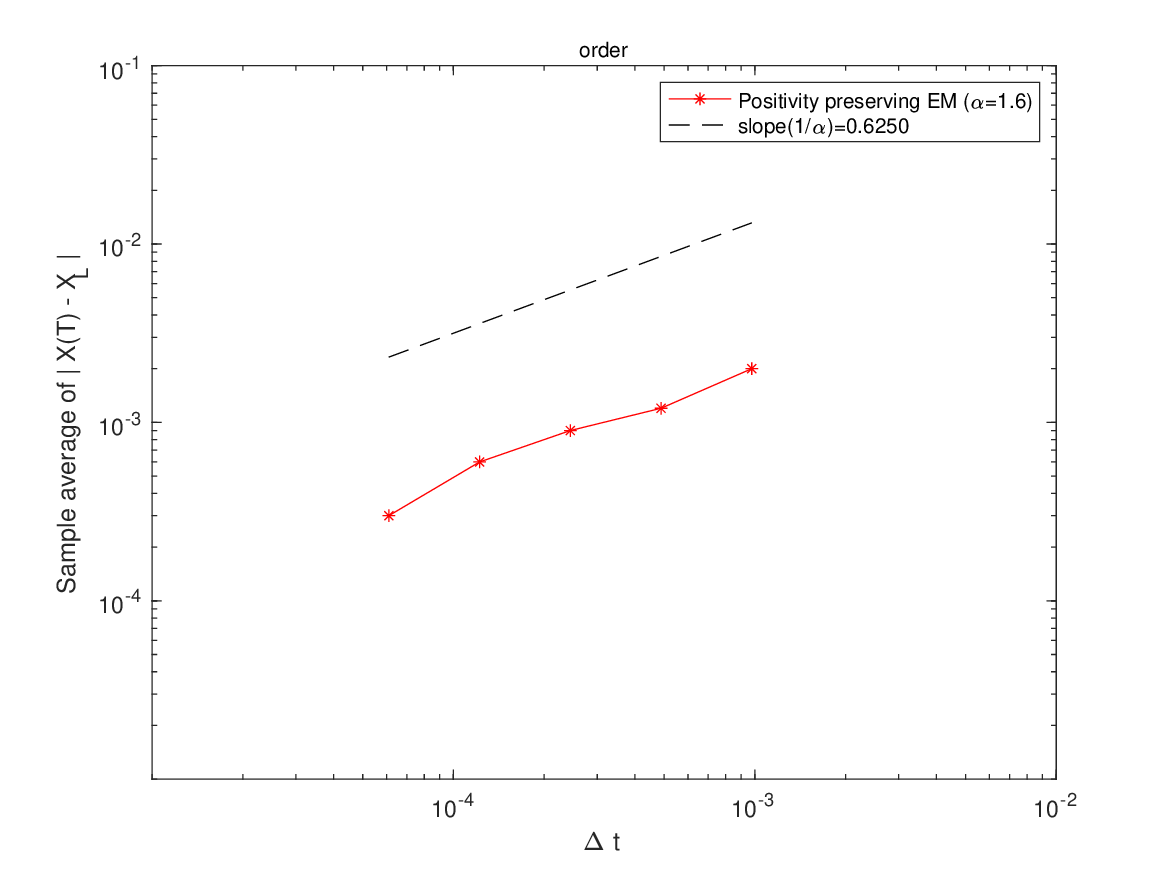}}
	\subfigure[{\label{fig:3_2}} $\alpha=1.4$]
	{\includegraphics[width=0.49\linewidth]{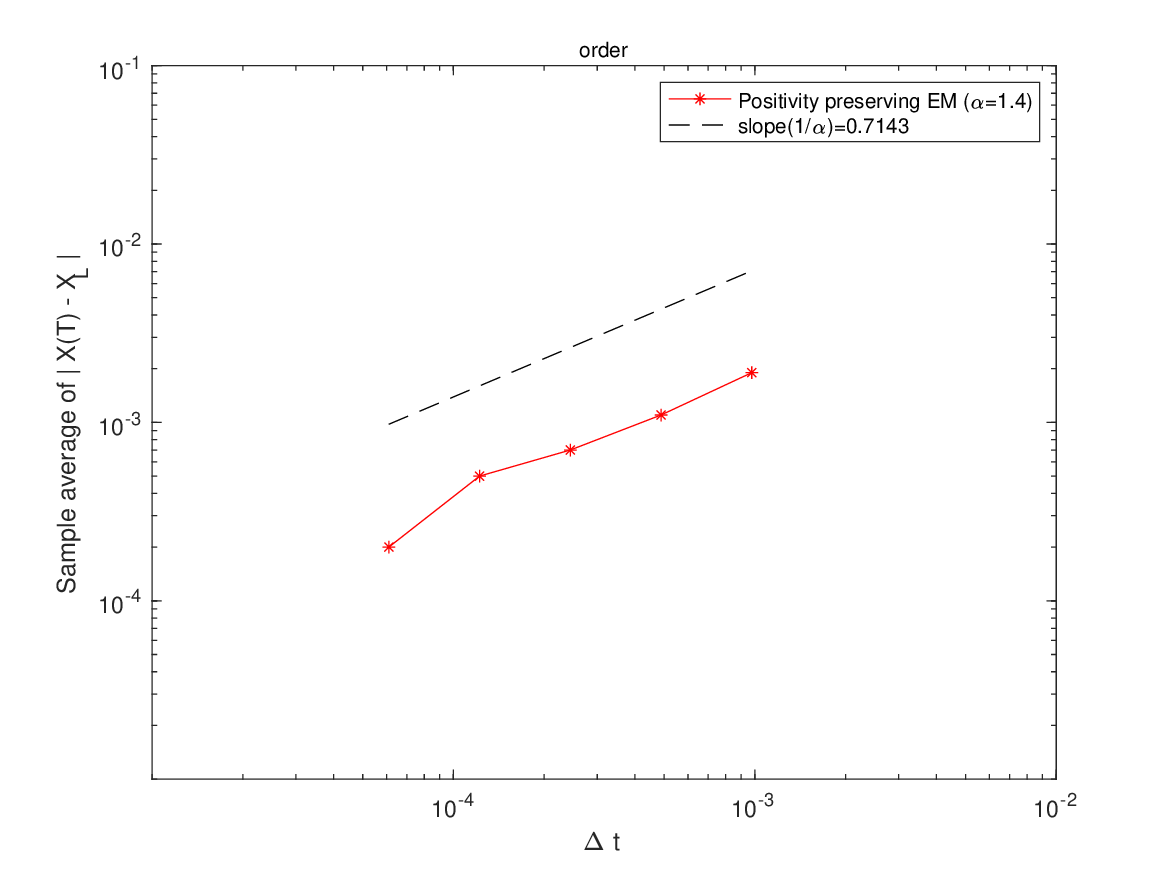}}
	\subfigure[{\label{fig:4_2}} $\alpha=1.1$]
	{\includegraphics[width=0.49\linewidth]{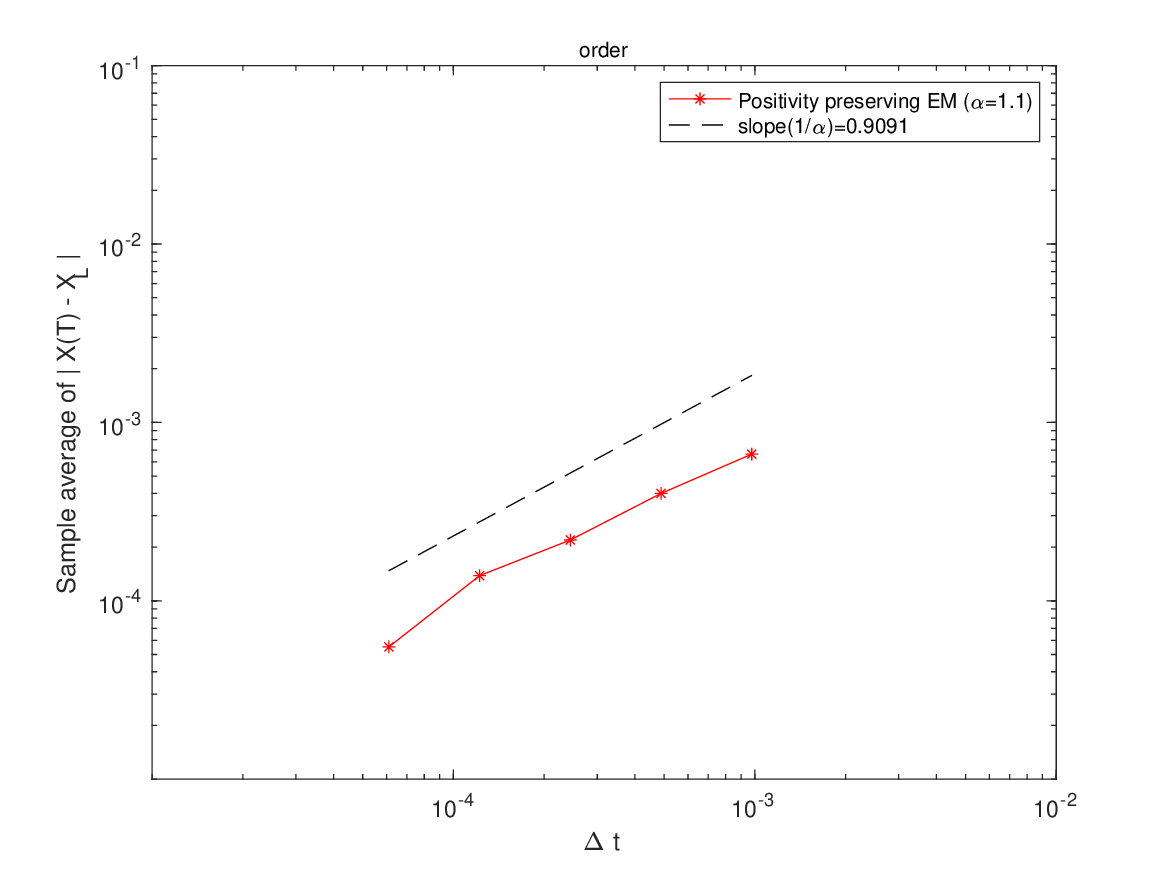}}
	\caption{Errors versus stepsize $\Delta $ on log-log scale with $\mu=2$, $\lambda=3$, $\kappa=0.5$}
	\label{ex12}
\end{figure}

\begin{table}[htph]
	\centering
	\renewcommand{\arraystretch}{1.2}
	\setlength{\tabcolsep}{11pt}
	\caption{Numerical results for \eqref{exsde:1} with $\mu=2$, $\lambda=3$, $\kappa=0.2$. }
	\label{table3}
	\begin{tabular}{ccccc}
		\Xhline{2\arrayrulewidth}
		$\Delta $ & $\alpha=1.8$   &  $\alpha=1.6$ &  $\alpha=1.4$  &  $\alpha=1.1$ \\ \Xhline{2\arrayrulewidth}
		$2^{-14}$ &  $0.0731\times10^{-3}$  &  $0.0583\times10^{-3}$ & $0.0261\times10^{-3}$ & $0.0133\times10^{-3}$ \\
		$2^{-13}$ &  $0.1012\times10^{-3}$  &  $0.0840\times10^{-3}$ & $0.0555\times10^{-3}$ & $0.0245\times10^{-3}$ \\
		$2^{-12}$ & $0.1841\times10^{-3}$  &  $0.1329\times10^{-3}$ & $0.0819\times10^{-3}$ & $0.0523\times10^{-3}$ \\
		$2^{-11}$ &  $0.2629\times10^{-3}$  &  $0.2296\times10^{-3}$ & $0.1516\times10^{-3}$ & $0.0816\times10^{-3}$ \\
		$2^{-10}$ &  $0.3504\times10^{-3}$  &  $0.3118\times10^{-3}$ & $0.2510\times10^{-3}$ & $0.1685\times10^{-3}$ \\ 
		\Xhline{2\arrayrulewidth}
	\end{tabular}
\end{table}

\begin{figure}[H]
	\centering
	\subfigure[{\label{fig:1_3}} $\alpha=1.8$]
	{\includegraphics[width=0.49\linewidth]{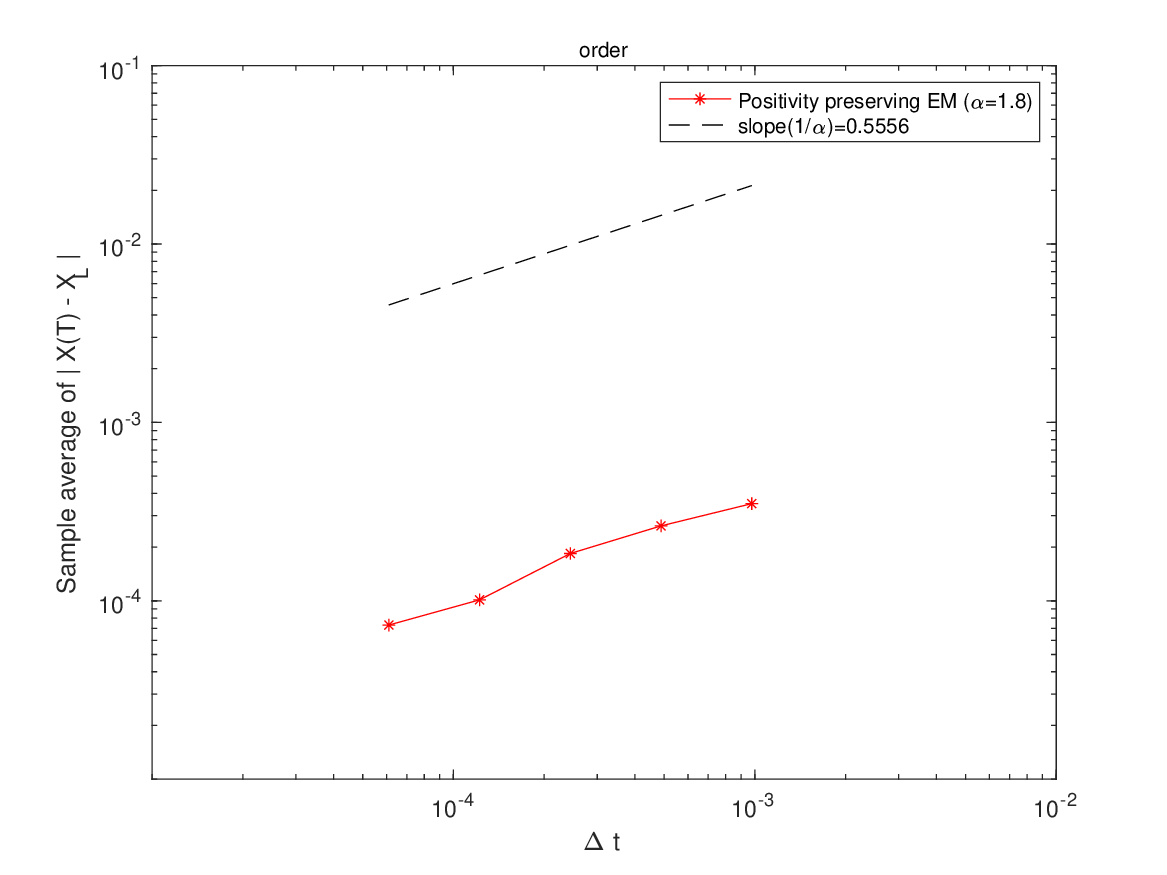}}
	\subfigure[{\label{fig:2_3}} $\alpha=1.6$]
	{\includegraphics[width=0.49\linewidth]{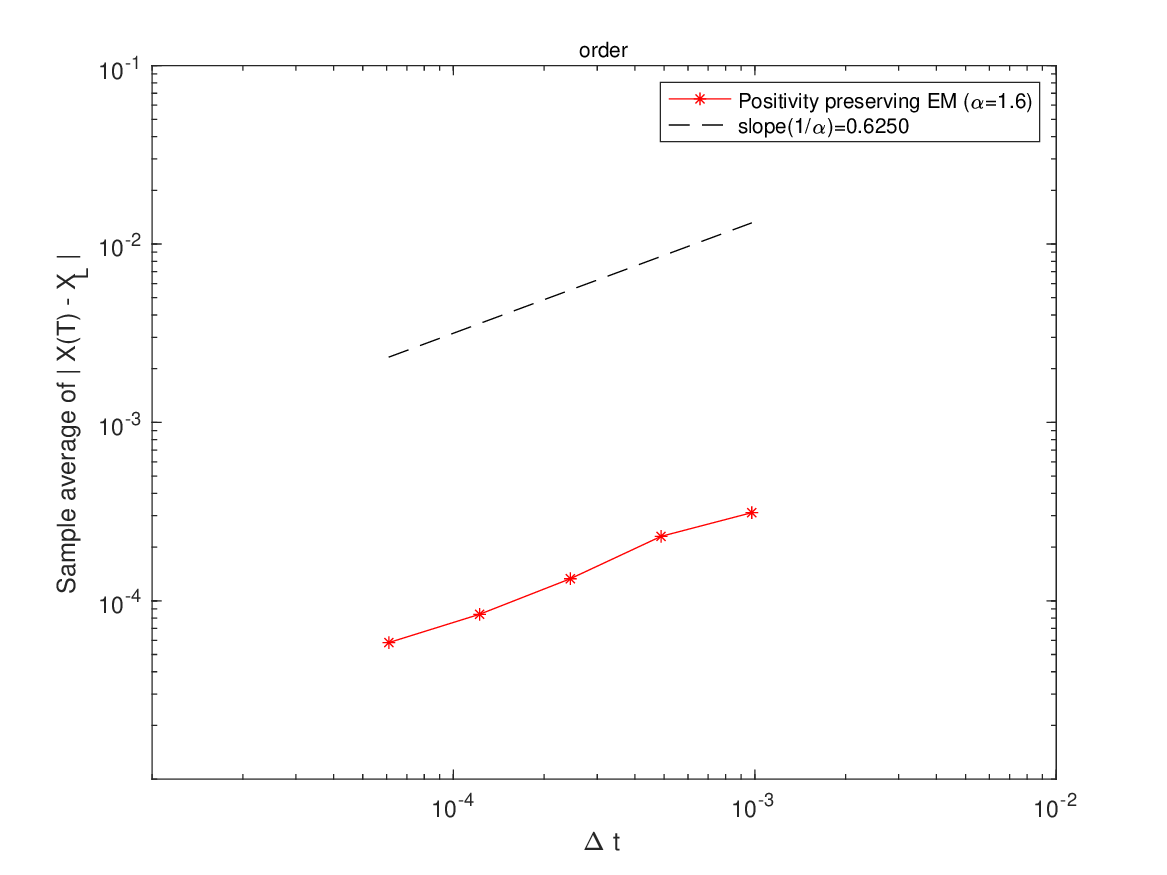}}
	\subfigure[{\label{fig:3_3}} $\alpha=1.4$]
	{\includegraphics[width=0.49\linewidth]{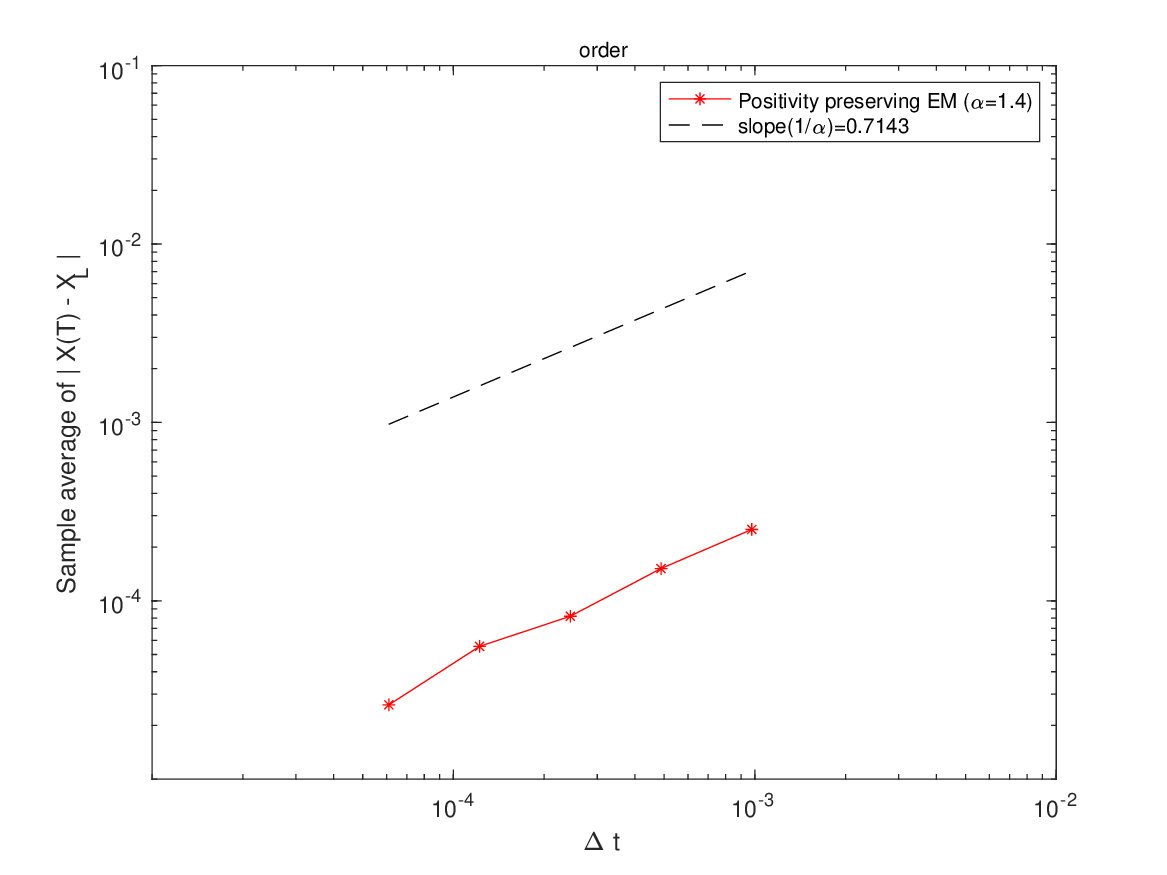}}
	\subfigure[{\label{fig:4_3}} $\alpha=1.1$]
	{\includegraphics[width=0.49\linewidth]{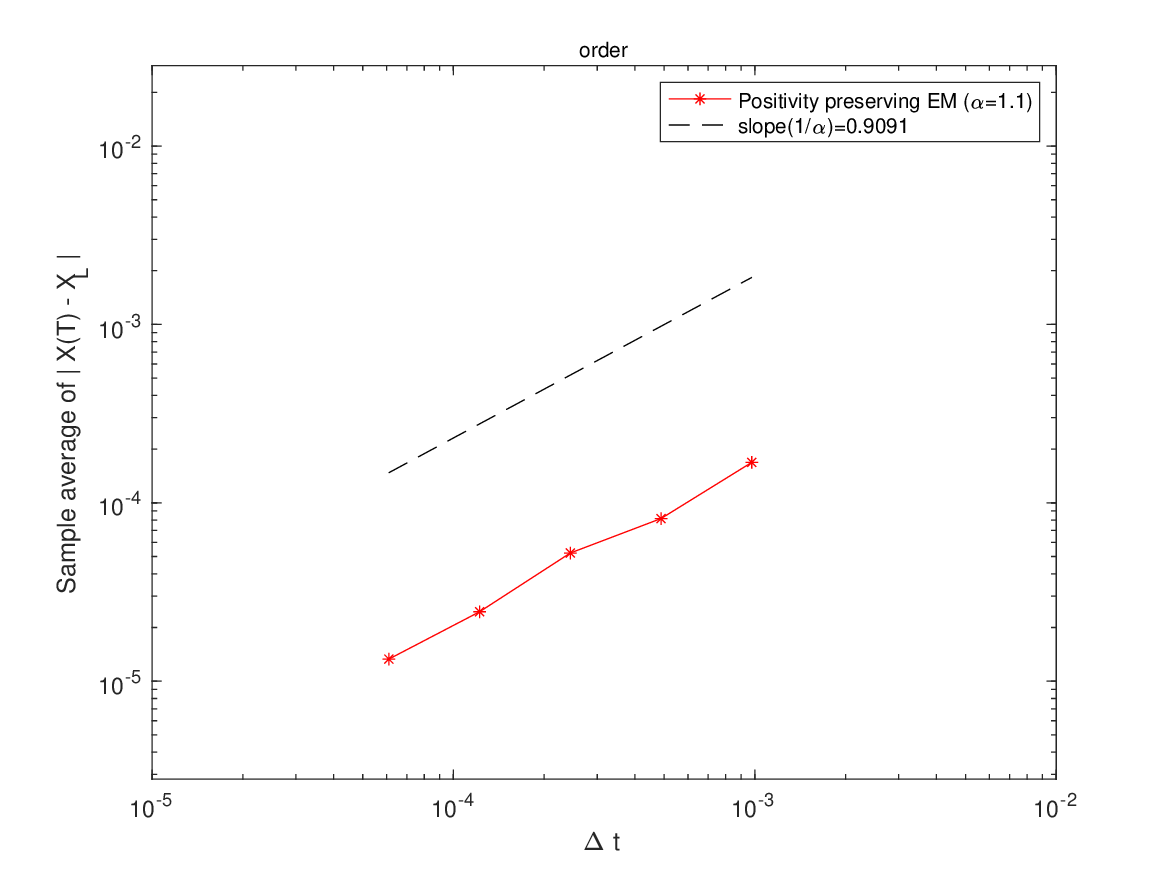}}
	\caption{Errors versus stepsize $\Delta $ on log-log scale with $\mu=2$, $\lambda=3$, $\kappa=0.2$}
	\label{ex13}
\end{figure}

\section{Conclusion and future work}\label{sec6}
In this paper, we investigate the numerical method for financial models driven by $\alpha$-stable processes. First, we prove that under certain conditions, the model admits a unique global positive solution. Next, we develop a positivity preserving EM method, prove its convergence, and show that the convergence rate is  $1/\alpha$.

There are several important problems that merit further exploration. One challenging aspect for theoretical analysis is the presence of super-linear terms in the coefficients of financial models. These terms, which grow faster than linearly, introduce significant complexity in both the analysis and the numerical methods. Another interesting area for future research is the long-term behavior of these models, as understanding how they evolve over time could provide valuable insights.

\end{document}